\documentclass[a4paper,reqno,oneside]{amsart}
\usepackage{enumitem,amssymb}
\usepackage{dsfont,graphicx}
\usepackage{url}

\newtheorem{thm}{Theorem}
\newtheorem{lem}[thm]{Lemma}
\newtheorem{prop}[thm]{Proposition}
\newtheorem{corol}[thm]{Corollary}
\theoremstyle{definition}
\newtheorem{defin}[thm]{Definition}
\newtheorem{exm}[thm]{Example}
\newtheorem{rem}[thm]{Remark}
\newtheorem{conj}[thm]{Conjecture}

\newcommand{\pv}{{\rm p.v.}}

\newcommand{\rmi}{\mathrm{i}}
\newcommand{\C}{\mathbb{C}}
\newcommand{\R}{\mathbb{R}}

\newcommand{\T}{\mathbb{T}}
\newcommand{\D}{\mathbb{D}}
\newcommand{\N}{\mathbb{N}}
\newcommand{\Z}{\mathbb{Z}}
\newcommand{\cA}{\mathcal{A}}

\newcommand{\cH}{\mathcal{H}}
\newcommand{\cD}{\mathcal{D}}

\newcommand\ps[2]{\langle #1, #2 \rangle}

\DeclareMathOperator{\dom}{dom}
\DeclareMathOperator{\ran}{ran}

\title[]{A variational formulation for Dirac operators in bounded domains. Applications to spectral geometric inequalities.}

\author{Pedro R.~S.~Antunes}
\address{(P.~R.~S.~ Antunes) Sec\c{c}\~{a}o de Matem\'atica, Departamento de Ci\^{e}ncias e Tecnologia, Universidade Aberta, Pal\'acio Ceia, 1269-001 Lisbon, Portugal, and Grupo de F\'{i}sica Matema\'{a}tica, Faculdade de Ci\^{e}ncias, Universidade de Lisboa, Campo Grande, Edif\'{i}cio C6, P-1749-016 Lisboa, Portugal}
\email{prantunes@fc.ul.pt}
\urladdr{http://webpages.ciencias.ulisboa.pt/~prantunes/}

\author{Rafael D.~Benguria}
\address{(R.~D.~Benguria) Instituto de F\'isica, Pontificia Universidad Cat\'olica de Chile, Avda. Vicu\~{n}a Mackenna 4860, Santiago, Chile.}
\email{rbenguri@fis.puc.cl}
\urladdr{http://www.fis.puc.cl/~rbenguri/}

\author{Vladimir Lotoreichik}
\address{(V.~Lotoreichik) Department of Theoretical Physics, Nuclear Physics Institute, Czech Academy of Sciences, 25068 \v{R}e\v{z}, Czech Republic.}
\email{lotoreichik@ujf.cas.cz}
\urladdr{http://gemma.ujf.cas.cz/~lotoreichik/}

\author{Thomas Ourmi\`eres-Bonafos}
\address{(T.~Ourmi\`eres-Bonafos) Aix Marseille Univ, CNRS, Centrale Marseille, I2M, Marseille, France}
\email{thomas.ourmieres-bonafos@univ-amu.fr}
\urladdr{http://www.i2m.univ-amu.fr/perso/thomas.ourmieres-bonafos/}

\begin{document}


\begin{abstract} We investigate spectral features of the Dirac operator with infinite mass boundary conditions in a smooth bounded domain of $\mathbb{R}^2$. Motivated by spectral geometric inequalities, we prove a non-linear variational formulation to characterize its principal  eigenvalue. This characterization turns out to be very robust and allows for a simple proof of a Szeg\"o type inequality as well as a new reformulation of a Faber-Krahn type inequality for this operator. The paper is complemented with strong numerical evidences supporting the existence of a Faber-Krahn type inequality.
\end{abstract}                                                            

\maketitle
\tableofcontents
\section{Introduction}
\subsection{Motivations and state of the art}
In the past few years there has been a growing interest in the study of Dirac operators among the mathematical physics community; the main reason being that low-energy electrons in a single-layered sheet of graphene are driven by an effective hamiltonian being a two-dimensional massless Dirac operator.

Various mathematical studies have been undertaken, starting with a rigorous mathematical derivation of such hamiltonians, see \emph{e.g.} \cite{FW12} for the effective hamiltonian derivation or \cite{ALTMR,BCLTS,MOBP,SV18} for the justification of the so-called infinite mass boundary conditions. Many properties of such operators have been investigated as their self-adjointness in bounded domains with specified boundary conditions or coupled with the so-called $\delta$-interactions, see \cite{BHPOB,BFVdBS17}. Let us also mention recent works on spectral properties and asymptotics of Dirac-type operators in specific asymptotic regimes (see \cite{ALTR,HOBP}).

In this work, we are interested in finding geometrical bounds on the eigenvalues of one of the simplest Dirac operator relevant in physics: the two-dimensional massless Dirac operator with infinite mass boundary conditions.

To set the stage, let $\Omega \subset \R^2$ be a $C^\infty$ simply connected domain and let $n = (n_1,n_2)^\top$ be the outward pointing normal field on $\partial\Omega$.  The Dirac operator with infinite mass boundary conditions in $L^2(\Omega,\C^2)$ is defined as
\begin{multline*}
	D^\Omega  := \begin{pmatrix}
					0 & -2\rmi\partial_z\\
					-2\rmi\partial_{\bar z} & 0
				\end{pmatrix},\\\dom(D^\Omega) := \{ u = (u_1,u_2)^\top \in H^1(\Omega,\C^2) : u_2 = \rmi {\bf n}u_1 \text{ on }\partial\Omega \},
\end{multline*}
where we have set ${\bf n} := n_1 + \rmi n_2$ and with the Wirtinger operators defined as usual by
\[
	\partial_z = \frac12(\partial_1 - \rmi \partial_2),\quad \partial_{\bar{z}} = \frac12(\partial_1 + \rmi \partial_2).
\]

The Dirac operator with infinite mass boundary conditions $D^\Omega$ is known to be self-adjoint (see \cite[Thm. 1.1.]{BFVdBS17}), moreover its spectrum is symmetric with respect to the origin and constituted of eigenvalues of finite multiplicity satisfying
\[
	\cdots \leq -E_k(\Omega) \leq\cdots \leq-E_{1}(\Omega) < 0 < E_{1}(\Omega) \leq \cdots \leq E_k(\Omega) \leq \cdots.
\]
In the recent paper \cite{BFVdBS17b}, the following geometrical lower bound is obtained
\begin{equation}\label{eqn:lbBFSVdB}
	E_1(\Omega) \geq \sqrt{\frac{2\pi}{|\Omega|}},
\end{equation}
where $|\Omega|$ denotes the area of the domain $\Omega$. However, this lower bound is never attained among Euclidean domains and by analogy with the famous Faber-Krahn inequality \cite{F23,K25}, a natural conjecture for the optimal lower-bound is the following.
\begin{conj}\label{conj:FK} There holds
\[
	E_1(\Omega) \geq \sqrt{\frac{\pi}{|\Omega|}} E_1(\D),
\]
where $\D$ is the unit disk. There is equality in the above inequality if and only if $\Omega$ is a disk.
\end{conj}
\begin{rem} As explained in \cite[Remark 2]{BFVdBS17b} (see also \cite[Appendix]{LOB}), the eigenstructure of the unit disk is explicit. Indeed, $E_1(\D) \simeq 1.435\dots$ is the first non-negative root of the equation $J_0(E) = J_1(E)$ where $J_0$ and $J_1$ are the Bessel functions of the first kind of order $0$ and of order $1$, respectively. Moreover, an associated eigenfunction is given for $(x_1,x_2)\in \D$ by
\[
	\begin{pmatrix}J_0(|x|)\\ \rmi\frac{x_1 + \rmi x_2}{|x|}J_1(|x|)\end{pmatrix}.
\]
\label{rem:fundisk}
\end{rem}
Conjecture \ref{conj:FK} motivated part of this paper and is still an open question. However, in Section \ref{sec:numerics} we provide strong numerical evidences supporting it and in Section \ref{sec:aboutFK} we show how Conjecture \ref{conj:FK} is intimately connected to the famous Bossel-Daners inequality for the Robin Laplacian (see \cite{Boss86,Dan06}).

The quest for a geometrical upper-bound has also attracted attention recently as for instance in \cite{LOB}. In this work, the given geometrical upper-bound is sharp in the sense that it is an equality if and only if the considered domain is a disk. Nevertheless, this upper-bound depends in a complicated fashion of different geometrical parameters and may be hard to compute in practice.

Let us also mention that similar questions are dealt with in the differential geometry literature for lower bounds and upper bounds for Dirac operators on spin-manifolds (see for instance \cite{AF99,B92,B98,R06}).

One of the main result of this paper is the following theorem which gives a geometrical upper-bound in term of simple geometric quantities: $|\Omega|$ the area of $\Omega$, $|\partial\Omega|$ the perimeter of $\Omega$ as well as $r_i$ the inradius of $\Omega$.

\begin{thm} Let $\Omega \subset \R^2$ be a $C^\infty$ simply connected domain. There holds
\[
	E_1(\Omega) \leq \frac{|\partial\Omega|}{(\pi r_i^2 + |\Omega|)}E_1(\D),
\]
with equality if and only if $\Omega$ is a disk.
\label{thm:ineq}
\end{thm}

The proof is by combining a new variational characterization of $E_1(\Omega)$, inspired by min-max techniques for operators with gaps introduced in \cite{DES00} and the classical proof of Szeg\"o about the eigenvalues of membranes of fixed area \cite{Sze54}.

It turns out this new variational characterization is of interest by itself because it also allows for numerical simulations and we believe that it could be an adequate starting point to prove Conjecture \ref{conj:FK} as discussed further on in Section \ref{sec:aboutFK}. To introduce it, consider the quadratic form
\begin{equation}
	q_{E,0}^\Omega(u) := 4 \int_\Omega |\partial_{\bar z} u|^2 dx - E^2 \int_{\Omega}|u|^2dx + E \int_{\partial\Omega} |u|^2 ds, \quad \dom(q_{E,0}^\Omega) := C^\infty(\overline{\Omega},\C).
\label{eqn:deffq}
\end{equation}
For $E>0$, $q_{E,0}$ is bounded below with dense domain and we consider $q_E^\Omega$ the closure in $L^2(\Omega)$ of $q_{E,0}^\Omega$. Then, we define the first min-max level
\begin{equation}
	\mu^\Omega(E) := \inf_{u \in \dom(q_E^\Omega)\setminus\{0\}} \frac{4 \int_\Omega |\partial_{\bar z} u|^2 dx - E^2 \int_{\Omega}|u|^2dx + E \int_{\partial\Omega} |u|^2 ds}{\int_\Omega |u|^2 dx}.
\label{eqn:firstminmax}
\end{equation}
The second main result of this paper is the following non-linear variational characterization of $E_1(\Omega)$.
\begin{thm} $E>0$ is the first non-negative eigenvalue of $D^\Omega$ if and only if $\mu^\Omega(E) = 0$.
\label{thm:vf}
\end{thm}
The advantage of the quadratic form $q_E^\Omega$ is two-fold. First, functions in the considered variational space are now scalar valued and, second, the infinite mass boundary conditions does not appear in the variational formulation. However, the first drawback is that $\dom{(q_E^\Omega)}$ contains the Hardy space $\cH_{\rm h}^2(\Omega)$, constituted of holomorphic functions with traces in $L^2(\partial\Omega)$. In particular, $\dom{(q_E^\Omega)}$ is not a usual Sobolev space and a special care is needed in order to prove Theorem \ref{thm:vf}. In particular, it asks for a precise description of the domain $\dom{(q_E^\Omega)}$ as well as the domain of the associated self-adjoint operator \emph{via} Kato's first representation theorem (see \cite[Chap. VI, Thm. 2.1]{Kat}). It is done using convolution operators reminiscent of what is done in \cite{AMV,OBV}, elliptic regularity properties of the maximal Wirtinger operators as well as using Cauchy singular integral operators on $\partial\Omega$, seen as periodic pseudo-differential operators.\\

Theorem \ref{thm:vf} is reminiscent of \cite{DES00,DES03}, where a similar strategy is used to deal with the Dirac-Coulomb operator. To our knowledge, this is the first time this idea is extended to boundary value problems and now, we describe its heuristic.

Let $(u,v)^\top \in \dom{(D^\Omega)}$ be an eigenfunction associated with the eigenvalue $E>0$. In $\Omega$, the eigenvalue equation reads
\begin{equation}\label{eqn:vpdir}
	-2\rmi\partial_z v = E u,\quad -2\rmi\partial{\bar z} u = E v.
\end{equation}
If we assume that this identity is true up to the boundary $\partial\Omega$, we obtain the following boundary condition for $u$:
\begin{equation}\label{eqn:bc1}
	\overline{n} \partial_{\bar z} u + \frac{E}2 u = 0 \text{ on } \partial\Omega.
\end{equation}
Now, Equation \eqref{eqn:vpdir}  gives
\begin{equation}\label{eqn:eqvpdeduced}
		-4 \partial_z \partial_{\bar z} u = E^2 u \text{ in } \Omega.
\end{equation}
Hence, a weak formulation is obtained taking the scalar product by $u$, integrating by parts and taking into account the boundary condition \eqref{eqn:bc1}. This formally gives $q_E^\Omega(u) = 0$ and this is the reason for introducing the quadratic form $q_E^\Omega$ in \eqref{eqn:deffq}.

Let us add two remarks. The first one explains that \eqref{eqn:bc1}-\eqref{eqn:eqvpdeduced} can be recast into a non-linear eigenvalue problem for a Laplace operator with oblique boundary conditions. The second remark, explains how Theorem \ref{thm:vf} could be extended to handle the next eigenvalues.
\begin{rem}Note that \eqref{eqn:eqvpdeduced} is an eigenvalue equation for the Laplace operator and reads $-\Delta u = E^2 u$. The boundary condition \eqref{eqn:bc1} is a relation between the normal derivative, the tangential derivative and the value of the function on $\partial\Omega$. If we let ${\bf t}$ be the tangent field on $\partial\Omega$ such that $({\bf n},{\bf t})$ is a direct frame, the problem can be re-interpreted as an oblique problem
\begin{equation}
	\label{eqn:vpnl}
	\left\{\begin{array}{ccll}
			-\Delta u & = & E^2 u &\text{ in } \Omega,\\
			 \partial_{n}u + \rmi \partial_t u+ E u &=& 0 &\text{ on } \partial\Omega,
		\end{array}
	\right.
\end{equation}
where $\partial_n$ and $\partial_t$ are the normal and tangential derivatives, respectively.

Note that Problem \eqref{eqn:vpnl} is non-linear because the parameter $E>0$ appears both in the eigenvalue equation and in the boundary condition.
\end{rem}

\begin{rem} For $j\geq1$, one can consider the $j$-th min-max level of $q_E^\Omega$ defined as
\[
	\mu_j^\Omega(E) := \inf_{\tiny{\begin{array}{c}F \subset \dom{(q_E^\Omega)}\\ \dim F = j\end{array}}}\sup_{u \in F\setminus\{0\}} \frac{4 \int_\Omega |\partial_{\bar z} u|^2 dx - E^2 \int_{\Omega}|u|^2dx + E \int_{\partial\Omega} |u|^2 ds}{\int_\Omega |u|^2 dx}.
\]
As in \cite{DES00}, Theorem \ref{thm:vf} could be extended as follows: $E>0$ is the $j$-th non-negative eigenvalue of $D^\Omega$ if and only if $\mu_j^\Omega(E) = 0$. We do not discuss it here because we are concerned only with the principal eigenvalue $E_1(\Omega)$.
\end{rem}

Finally, let us comment the hypothesis on $\Omega$. First, one would like to lower the smoothness hypothesis to be able to handle, for instance, Lipschitz domains. This is a natural question but there is no reason for the Dirac operator with infinite mass boundary  to be self-adjoint on such a domain $\dom{(D^\Omega)}$ (see the case of polygonal domains in \cite{LTOB}). Moreover, as part of the proof relies on pseudo-differential techniques, we prefer to keep the $C^\infty$ smoothness assumption on $\partial\Omega$ because it allows for a more efficient treatment of singular integral operators on the boundary. Second, the simply connectedness assumption may be an unnecessary hypothesis for Theorem \ref{thm:vf} to hold. Nevertheless, we are not able to drop it in Theorem \ref{thm:ineq} because the proof relies on the Riemann mapping theorem to build an admissible test function for $q_E^\Omega$.

\subsection{Structure of the paper}
In Section \ref{sec:prelim}, we gather several results on Sobolev spaces on $\partial\Omega$, periodic pseudo-differential operators on $\partial\Omega$ and deduce various mapping properties of the Cauchy singular integral operators.

Section \ref{sec:maxc-rop} contains a description of the domain of the maximal Wirtinger operators. In particular, we discuss the existence of a trace operator for functions belonging to these domains and state a fundamental elliptic regularity result.

Section \ref{sec:berghard} deals with the description of the Bergman and Hardy spaces on $\Omega$ thanks to integral operators. This is done by introducing the Szeg\"o projectors on the Sobolev spaces on the boundary $H^s(\partial\Omega)$ ($s\in\{-\frac12,0,\frac12\}$). As a byproduct of this analysis we are able to describe explicitly the domains of the maximal Wirtinger operators.

Theorem \ref{thm:vf} is proved in Section \ref{sec:vcproof}. We start by describing the domain of the quadratic form $q_E^\Omega$ in terms of the first-order Sobolev space $H^1(\Omega)$ and the Hardy space on $\Omega$. Then, the analysis is pushed forward to study the domain of the self-adjoint operator associated with $q_E^\Omega$ \emph{via} Kato's first representation theorem (see \cite[Chap. VI, Thm. 2.1]{Kat}). Combining these tools, we prove Theorem \ref{thm:vf}.

Then, we apply Theorem \ref{thm:vf} in Section \ref{sec:isopin} to prove Theorem \ref{thm:ineq}. The proof is by adapting the well-known proof of Szeg\"o \cite{Sze54} to our setting, constructing an adequate test function for the new variational formulation.

In Section \ref{sec:aboutFK}, we show that Conjecture \ref{conj:FK} can be reformulated and that it is related to the famous Bossel-Daners inequality.

We conclude in Section \ref{sec:numerics} illustrating by numerical experiments the validity of Conjecture \ref{conj:FK} and several theoretical results discussed all along the paper.

\section{Preliminaries}\label{sec:prelim}
\subsection{Sobolev spaces on $\partial\Omega$}
In the following, $\T$ is the torus $\T := \mathbb{R}/\mathbb{Z}$, $\cD(\T) = C^\infty(\T)$ is the space periodic smooth functions on the torus $\T$ and $\cD(\T)'$ the space of periodic distributions on the torus $\T$.
Let  $f \in \cD(\T)'$ we define its Fourier coefficients using the duality pairing by
\[
	\widehat{f}(n) := \langle f, e_{-n}\rangle_{\cD(\T)',\cD(\T)},\quad e_n := t\in \T \mapsto e^{2\rmi\pi n t}.
\]
For $s \in \R$, the Sobolev space of order $s$ on $\T$ is defined as
\[
	H^s(\T) := \{f \in \cD(\T)' : \sum_{n =-\infty}^{+\infty}(1+|n|)^{2s}|\widehat{f}(n)|^2 < +\infty\}.
\]
Set $\ell := |\partial\Omega|$ and let $\gamma : \R \big/ [0,\ell] \to \partial\Omega$ be a smooth arc-length parametrization of $\partial\Omega$. Consider the map
\[
	U^* : \cD(\T) \to \cD(\partial\Omega),\quad (U^*g)(x) := \ell^{-1}g(\ell^{-1}\gamma^{-1}(x)),\ x\in\partial\Omega,
\]
where we have set $\cD(\partial\Omega) := C^\infty(\partial\Omega)$. We define the map $U : \cD(\partial\Omega)' \to \cD(\T)'$ as
\begin{equation}\label{eqn:defU}
	\langle U f, g\rangle_{\cD(\T)',\cD(\T)} := \langle  f, U^*g\rangle_{\cD(\partial\Omega)',\cD(\partial\Omega)}.
\end{equation}
The Sobolev space of order $s\in\R$ on $\partial\Omega$ is defined as
\[
	H^s(\partial\Omega) := \{ f \in \cD(\partial\Omega)' : Uf \in H^s(\T)\}.
\]

\subsection{Periodic pseudo-differential operators}
Let us start by defining  periodic pseudo-differential operators on $\T$.
\begin{defin} A linear operator $H$ on $C^\infty(\T)$ is a periodic pseudo-differential operator on $\T$ if there exists $h : \T \times \Z \to \C$ such that:
\begin{enumerate}
	\item for all $n\in \Z$, $h(\cdot,n) \in C^\infty(\T)$,
	\item $H$ acts as $Hf = \sum_{n\in\Z} h(\cdot,n) \widehat{f}(n) e_n$,
	\item there exists $\alpha \in \R$ such that for all $p,q\in\N_0$ there exists $c_{p,q}>0$ such that there holds
	\[
		\Big|\big(\frac{d^p}{dt^p} (\omega^q h)\big)(t,n)\Big| \leq c_{p,q}(1+|n|)^{\alpha-q},
	\]
where the operator $\omega$ is defined for all $(t,n)\in\T\times\Z$ by $(\omega h)(t,n) := h(t,n+1) - h(t,n)$.
\end{enumerate}
$\alpha$ is called the order of the pseudo-differential operator $H$. The set of pseudo-differential operators of order $\alpha$ on $\T$ is denoted $\Psi^\alpha$ and we define
\[
	\Psi^{-\infty} := \bigcap_{\alpha\in\R} \Psi^{\alpha}.
\]
\end{defin}

\begin{exm} For further use, we introduce the example of multiplication operators. Consider $H : \C^\infty(\T) \to C^\infty(\T)$ defined as
\[
	(Hf)(t) := h(t) f(t),\quad h\in C^\infty(\T).
\]
Decomposing in Fourier series, one immediately obtains
\[
	(Hf) = \sum_{n\in\Z} h\widehat{f}(n) e_n.
\]
There holds $\omega^q h =0$ for all $q\geq1$ and, as $h\in C^\infty(\T)$, for all $t\in\T$ we obtain
\[
	\Big|\big(\frac{d^p h}{dt^p}\big)(t)\Big| \leq c_{p},\quad\text{for some } c_p>0
\]
and we get $H \in \Psi^0$.
\label{exm:pseudo0}
\end{exm}

Using the map $U$ defined in \eqref{eqn:defU}, we define periodic pseudo-differential operators on $\partial\Omega$ as follows.
\begin{defin} A linear operator $H$ on $C^\infty(\partial\Omega)$ is a periodic pseudo-differential operator on $\partial\Omega$ of order $\alpha\in \R$ if the operator $H_0 := U H U^{-1} \in \Psi^\alpha$. The set of pseudo differential operators on $\partial\Omega$ of order $\alpha$ is denoted $\Psi^\alpha_{\partial\Omega}$ and we set
\[
	\Psi_{\partial\Omega}^{-\infty} := \bigcap_{\alpha\in\R}\Psi_{\partial\Omega}^\alpha.
\]
\end{defin}
We will need the following properties of pseudo-differential operators on $\partial\Omega$. They can be found in \cite[\S 5.8 \& 5.9]{SV}.
\begin{prop}\label{prop:pseudoimpo} Let $s,\alpha,\beta\in \R$ and $H\in \Psi_{\partial\Omega}^\alpha, G \in \Psi_{\partial\Omega}^\beta$.
\begin{enumerate}
	\item $H$ extends uniquely to a bounded linear operator, also denoted $H$, from $H^s(\partial\Omega)$ to $H^{s-\alpha}(\partial\Omega)$.
	\item\label{itm:22} \label{itm:pseudoimp1}There holds
		\[
			H + G \in \Psi_{\partial\Omega}^{\max(\alpha,\beta)},\quad HG \in \Psi_{\partial\Omega}^{\alpha+\beta},\quad [H,G] \in \Psi_{\partial\Omega}^{\alpha+\beta -1}. 
		\]
\end{enumerate}
\end{prop}
\subsection{Cauchy singular integral operators}For $f \in C^\infty(\partial\Omega)$, the Cauchy singular integral operator is defined as a principal value by
\[
	S_{\rm h}(f)(z) := \frac1{i\pi} \pv \int_{\partial\Omega} \frac{f(\xi)}{\xi-z}d\xi,\quad z\in \partial\Omega.
\]
We define its anti-holomorphic counterpart as
\[
		S_{\rm ah}(f)(z) := \overline{S_{\rm h}(\overline{f})(z)} = -\frac1{i\pi} \pv \int_{\partial\Omega} \frac{f(\xi)}{\overline{\xi}-\bar z}d\overline{\xi},\quad z\in \partial\Omega.
\]
It turns out $S_{\rm h}$ and $S_{\rm ah}$ are periodic pseudo-differential operators on $\partial\Omega$. This is the purpose of the following proposition.
\begin{prop} The linear maps $S_{\rm h}$ and $S_{\rm ah}$ are periodic pseudo-differential operators of order $0$. In particular, they are bounded linear operators from $H^s(\partial\Omega)$ onto itself for all $s\in \R$.
\label{prop:mapcauchy}
\end{prop}
\begin{proof} This is proved in \cite[Prop 2.9.] {BHOBP19} where the operators $S_{\rm h}$ and $S_{\rm ah}$ are denoted $C_{\Sigma}$ and $-C_\Sigma'$ respectively (with $\Sigma := \partial\Omega$).
\end{proof}

We will also need the following property.
\begin{prop} Let  $H_\mathbf{n}$ be the multiplication operator by the normal $\mathbf{n}$ in $C^\infty(\partial\Omega)$. There holds:
\begin{enumerate}
	\item\label{itm:1pseudo} $H_{\mathbf{n}}$ is a periodic pseudo-differential operator of order $0$.
	\item\label{itm:2pseudo} Let $\sharp \in \{{\rm h},{\rm ah}\}$ we have $[H_{\mathbf{n}},S_\sharp] \in \Psi_{\partial\Omega}^{-1}$.
	\item\label{itm:3pseudo} There holds $S_{\rm ah} + S_{\rm h} \in \Psi_{\partial\Omega}^{-\infty}$.
\end{enumerate}
\label{prop:pseudocommut}
\end{prop}
\begin{proof} Point \eqref{itm:1pseudo} is proved remarking that the operator $U H_{\mathbf{n}}U^{-1}$ is a multiplication operator in $\T$. Thanks to Example \ref{exm:pseudo0}, we know that $U H_{\mathbf{n}}U^{-1} \in \Psi^0$ hence by definition we get $H_{\mathbf{n}}\in \Psi_{\partial\Omega}^0$.

Let us deal with Point \eqref{itm:2pseudo}. Let $\sharp \in \{{\rm h},{\rm ah}\}$, by Proposition \ref{prop:mapcauchy}, $S_\sharp \in \Psi_{\partial\Omega}^0$ and by Point \eqref{itm:1pseudo} $H_{\mathbf{n}} \in \Psi_{\partial\Omega}^0$. Hence, by \eqref{itm:pseudoimp1} Proposition \ref{prop:pseudoimpo}, we obtain Point \eqref{itm:2pseudo}.

Finally, we prove Point \eqref{itm:3pseudo}. By \cite[Proposition 2.9.]{BHOBP19} there exists $L\in \Psi_{\partial\Omega}^0$ and $R_1,R_2 \in \Psi_{\partial\Omega}^{-\infty}$ such that
\[
	S_{\rm h} = L + R_1,\quad S_{\rm ah} = -L + R_2. 
\]
Hence, $S_{\rm h} + S_{\rm ah} = R_1 + R_2 \in \Psi_{\partial\Omega}^{-\infty}$ by \eqref{itm:22} Proposition \ref{prop:pseudoimpo}.
\end{proof}
\section{Maximal Wirtinger operators}
\label{sec:maxc-rop}
In this section we describe elemental properties of the maximal Wirtinger operators defined as
\begin{align*}
		{\partial}_{\rm h} u =  \partial_{\bar z}u ,& \quad \dom( \partial_{\rm h}) := \{u \in L^2(\Omega) : \partial_{\bar z}u \in L^2(\Omega)\},\\
	{\partial}_{\rm ah} u = \partial_z u ,& \quad \dom( \partial_{\rm ah}) := \{u \in L^2(\Omega) : \partial_{z}u \in L^2(\Omega)\}.
\end{align*}

For $\sharp \in \{{\rm h}, {\rm ah}\}$, consider the operator norms $\|\cdot\|_{\sharp}$ defined as
\[
	\|u\|_{\sharp}^2 := \|\partial_{\sharp} u\|_{L^2(\Omega)}^2 + \|u\|_{L^2(\Omega)}^2,\quad u\in \dom(\partial_{\sharp}).
\]
In particular, $\dom(\partial_{\sharp})$ endowed with the scalar product defined for $u,v\in\dom(\partial_{\sharp})$ by
\[
\langle u,v\rangle_\sharp =\langle\partial_\sharp u,\partial_\sharp v\rangle_{L^2(\Omega)} + \langle u,v\rangle_{L^2(\Omega)}
\]
is a Hilbert space.

The first lemma is obtained by a simple integration by parts.
\begin{lem} The following identities hold.
\[
	H^1(\R^2) = \{f\in L^2(\R^2) : \partial_z f \in L^2(\R^2)\} = \{f\in L^2(\R^2) : \partial_{\bar z} f \in L^2(\R^2)\}
\]
\label{lem:equivfullspace}
\end{lem}
\begin{proof} Let $f\in C_0^\infty(\R^2)$. Integrating by parts several times we obtain:
\begin{align*}
	\|\nabla f\|_{L^2(\R^2)}^2 = \langle f, -\Delta f\rangle_{L^2(\R^2)} &= 4 \langle f, -\partial_{z}\partial_{\bar z} f\rangle_{L^2(\R^2)}  = 4 \|\partial_{\bar z} f\|_{L^2(\R^2)}^2\\
	&= 4 \langle f, -\partial_{\bar z}\partial_{z} f\rangle_{L^2(\R^2)}  = 4\|\partial_z f\|_{L^2(\R^2)}^2
\end{align*}
As $C_0^\infty(\R^2)$ is dense in $H^1(\R^2)$, we obtain the expected result.
\end{proof}

The next lemma is a density result.

\begin{lem} Let $\sharp \in \{{\rm h}, {\rm ah}\}$. The space $C^\infty(\overline{\Omega}) := C^\infty(\overline{\Omega},\C)$ is dense in $\dom(\partial_{\sharp})$.
\end{lem}
\begin{proof} Let $u \in \dom(\partial_{\rm h})$ and assume that for all $\varphi \in C^\infty(\overline{\Omega})$ there holds
\[
	0 = \ps{u}{\varphi}_{\rm h} = \ps{\partial_{\bar z} u}{\partial_{\bar z} \varphi}_{L^2(\Omega)} + \ps{u}{\varphi}_{L^2(\Omega)}.
\]
In particular, if $\varphi \in C_0^\infty(\Omega)$, we obtain $-\Delta u = - 4 u$ first in $\cD(\Omega)'$ then in $L^2(\Omega)$. Define $v = \partial_{\bar z} u$ and denote by $v_0$ its extension to the whole $\R^2$ by 0. For $\varphi \in C_0^\infty(\R^2)$ there holds
\begin{align*}
	\langle \partial_z v_0,\overline{\varphi}\rangle_{\mathcal{D}'(\R^2),\mathcal{D}(\R^2)} &= - \langle  v_0,\overline{\partial_{\bar z}\varphi}\rangle_{\mathcal{D}'(\R^2),\mathcal{D}(\R^2)}\\& = - \langle  v,\partial_{\bar z}\varphi\rangle_{L^2(\Omega)}\\& = - \langle  \partial_{\bar z} u,\partial_{\bar z}\varphi\rangle_{L^2(\Omega)}\\& = \langle  u,\varphi\rangle_{L^2(\Omega)}\\
& = \langle  u_0,\varphi\rangle_{L^2(\R^2)}\\
& = \langle u_0,\overline{\varphi}\rangle_{\mathcal{D}'(\R^2),\mathcal{D}(\R^2)},
\end{align*}
where $u_0$ denotes the extension by zero of $u$ to the whole $\R^2$. It gives $\partial_z v_0 = u_0 \in L^2(\R^2)$. By Lemma \ref{lem:equivfullspace}, $v_0$ is in $H^1(\R^2)$ and by \cite[Prop. IX.18.]{Bre} we get $v\in H_0^1(\Omega)$. Remark that in $\mathcal{D}'(\Omega)$, there holds $\partial_{\bar z}\partial_{ z} v = v$. Indeed, we have
\[
	\partial_{\bar z}\partial_{ z} v = \partial_{\bar z} \partial_{z}\partial_{\bar z} u = \partial_{\bar z} u = v.
\]
In particular this identity also holds true in $L^2(\Omega)$. Now, pick a sequence $v_n \in C_0^\infty(\Omega)$ converging to $v$ in the $H^1(\Omega)$-norm. There holds
\begin{align*}
	\langle v,v_n\rangle_{L^2(\Omega)} = \langle \partial_{z}\partial_{\bar z} v, v_n\rangle_{L^2(\Omega)} &= -\langle \partial_{\bar z} v, \overline{\partial_{\bar z}v_n}\rangle_{\mathcal{D}'(\Omega),\mathcal{D}(\Omega)}\\&=- \langle \partial_{\bar z} v, \partial_{\bar z}v_n\rangle_{L^2(\Omega)}
\end{align*}
Letting $n\to +\infty$ one obtains $\|v\|_{L^2(\Omega)}^2 = - \|\partial_{\bar z} v\|_{L^2(\Omega)}^2$ which implies $v = 0$. In $\mathcal{D}'(\Omega)$ we have $\partial_z v = \partial_z \partial_{\bar z} u  = u$. As $v=0$, $u=0$ which concludes the proof for $\sharp = {\rm h}$. The case $\sharp = {\rm ah}$ is handled similarly.
\end{proof}

In order to describe precisely the domains $\dom(\partial_\sharp)$ ($\sharp \in \{{\rm h}, {\rm ah}\}$) we need to prove the existence of traces on $\partial\Omega$ for functions in $\dom(\partial_\sharp)$. To this aim, define the following Dirichlet trace operators
\begin{equation}\label{eqn:deftrace}
	\Gamma^+ : H^1(\Omega) \to H^{\frac12}(\partial\Omega),\quad \Gamma^- : H_{loc}^1(\R^2\setminus\overline{\Omega}) \to H^{\frac12}(\partial\Omega).
\end{equation}
These linear operators are known to be bounded (see \cite[Thm. 3.37]{McLean}) and there exists continuous extension operators such that for $f\in H^{\frac12}(\partial\Omega)$ there holds
\[
	E^+f \in H^{1}(\Omega),\quad E^-f\in H^1(\R^2\setminus\overline{\Omega}) \quad\text{and}\quad \Gamma^\pm E^\pm f = f.
\]
Actually, the operator $\Gamma^+$ can be extended to functions in $\dom(\partial_\sharp)$ ($\sharp \in \{{\rm h},{\rm ah}\}$). This is the purpose of the following proposition.

\begin{lem}Let $\sharp \in \{{\rm h}, {\rm ah}\}$. The operator $\Gamma^+$ defined in \eqref{eqn:deftrace} extends into a linear bounded operator between $\dom(\partial_{\sharp})$ and $H^{-\frac12}(\partial\Omega)$.
\label{lem:extop}
\end{lem}

\begin{proof} Let $(v_n)_{n\in\N} \in C^\infty(\overline{\Omega})^\N$ be a sequence that converges to $v$ in the $\|\cdot\|_{\rm h}$-norm when $n\to +\infty$. Let us prove that $(\Gamma^+ v_n)_{n\in\N}$ has a limit in $H^{-\frac12}(\partial\Omega)$. First recall the integration by part formula
\[
	\frac12 \langle \Gamma^+u,\overline{\bf n} \Gamma^+w\rangle_{L^2(\partial\Omega)} = \langle \partial_{\bar z} u, w\rangle_{L^2(\Omega)} + \langle u,\partial_{z}w\rangle_{L^2(\Omega)}
\]
valid for any $u,w \in H^1(\Omega)$. Second, pick $f \in H^{\frac12}(\partial\Omega)$ and consider $w = E^+ ({\bf n} f) \in H^{1}(\Omega)$. There holds
\[
	\langle \Gamma^+(v_n - v_m),f\rangle_{L^2(\partial\Omega)} = 2 \langle \partial_{\bar z} (v_n - v_m), w\rangle_{L^2(\Omega)} + 2 \langle v_n-v_m,\partial_{z}w\rangle_{L^2(\Omega)}.
\]
In particular, we have
\begin{align*}
	\big|\langle \Gamma^+(v_n - v_m),f\rangle_{L^2(\partial\Omega)}\big| &\leq 2 \|\partial_{\bar z} (v_n - v_m)\|_{L^2(\Omega)}\|w\|_{L^2(\Omega)} + 2\|v_n-v_m\|_{L^2(\Omega)}\|\partial_z w\|_{L^2(\Omega)}\\
	& \leq 4 \|w\|_{H^1(\Omega)}\|v_n - v_m\|_{\rm h}\\
	& \leq 4 c_\Omega \|f\|_{H^{\frac12}(\partial\Omega)}\|v_n - v_m\|_{\rm h}\quad(\text{for some } c_\Omega > 0),
\end{align*}
where we have used that $E^+$ is a continuous linear map and that the multiplication operator by $\bf n$ is bounded from $H^{\frac12}(\partial\Omega)$ onto itself. When $n,m \to +\infty$ we obtain $\|\Gamma^+ (v_n - v_m)\|_{H^{-\frac12}(\partial\Omega)} \to 0$. In particular $(\Gamma^+v_n)_{n\in\N}$ is a Cauchy sequence in $H^{-\frac12}(\partial\Omega)$ thus converges to an element $g\in H^{-\frac12}(\partial\Omega)$ and we define $\Gamma^+ v := g$. Remark that the definition of $\Gamma^+v$ does not depend on the chosen sequence $(v_n)_{n\in\N}$ and that we have
\[
	\|\Gamma_+ v_n\|_{H^{-\frac12}(\partial\Omega)} \leq 4 c_\Omega \|v_n\|_{h}
\]
which implies, when $n\to +\infty$, that $\Gamma^+$ is bounded from $\dom(\partial_{\rm h})$ to $H^{-\frac12}(\partial\Omega)$. The proof for $\dom(\partial_{\rm ah})$ is handled similarly.
\end{proof}
\begin{rem}\label{rem:extop} If one picks $R>0$ such that $\overline{\Omega} \subset B(0,R) := \{x\in \mathbb{R}^2 : \|x\| < R\}$, one can prove that for $\star \in \{z,\overline{z}\}$, $\Gamma^-$ extends into a linear bounded operator between the space $\{u \in L^2(B(0,R)\setminus\Omega) : \partial_\star u \in L^2(B(0,R)\setminus\Omega)\}$ and $H^{-\frac12}(\partial\Omega)$. The proof goes along the same lines as the one of Lemma \ref{lem:extop}, using an extension operator $E^- : H^{\frac12}(\partial\Omega) \to H^1(B(0,R)\setminus\Omega)$ constructed such that for all $f\in H^{\frac12}(\partial\Omega)$, $E^-(f)|_{\partial B(0,R)} = 0$.
\end{rem}

\begin{rem} Pick $u\in\dom(\partial_{\rm ah})$ and $w\in H^1(\Omega)$. Note that by definition, the following Green's Formula holds
\begin{equation}
	\langle\partial_z u,w\rangle_{L^2(\Omega)} = - \langle u,\partial_{\bar z}w\rangle_{L^2(\Omega)} + \frac12\langle \overline{\bf n}\Gamma^+u,\Gamma^+w\rangle_{H^{-\frac12}(\partial\Omega),H^{\frac12}(\partial\Omega)}.
\label{eqn:Green}
\end{equation}

\end{rem}

The following elliptic regularity result is rather well known (see the analogous statement \cite[Lemma 2.4.]{BFVdBS17}).
\begin{lem}\label{lem:ellipregul} Let $\sharp \in \{{\rm h}, {\rm ah}\}$ and $u \in \dom(\partial_{\sharp})$. If $\Gamma^+ u \in H^{\frac12}(\partial\Omega)$ then $u \in H^1(\Omega)$. 
\end{lem}

\begin{proof}  Let $u \in \dom(\partial_{\rm h})$ be such that $\Gamma^+ u \in H^{\frac12}(\partial\Omega)$ and set $v = u - E^+(\Gamma^+ u)$. Then, $\Gamma^+ v = 0$ and if $v \in H_0^1(\Omega)$ the result is proved. If $v_n \in C^\infty(\overline{\Omega})$ is a sequence converging to $v$ in the $\|\cdot\|_{\rm h}$-norm there holds $\Gamma^+ v_n \to 0$ in $H^{-\frac12}(\partial\Omega)$ by Lemma \ref{lem:extop}. In particular, it gives for any $w \in H^1(\Omega)$
\begin{align*}
	\langle v, \partial_z w \rangle_{L^2(\Omega)} &= \lim_{n\to +\infty}\Big(- \langle \partial_{\bar z}v_n, w \rangle_{L^2(\Omega)} + \frac12 \langle \Gamma^+v_n, \overline{\bf n}\Gamma^+w \rangle_{L^2(\partial\Omega)}\Big)\\
	& = - \langle \partial_{\bar z}v, w\rangle_{L^2(\Omega)}.
\end{align*}
Let $v_0$ (resp. $h_0$) be the extension of $v$ (resp. $h := \partial_{\bar z}v$) by zero to the whole $\R^2$. If $\varphi \in C_0^\infty(\R^2)$, there holds
\begin{align*}
	- \langle h_0,\overline{\varphi}\rangle_{\mathcal{D}'(\R^2),\mathcal{D}(\R^2)} = - \langle h,\varphi\rangle_{L^2(\Omega)} &= \langle v,{\partial_{z}\varphi}\rangle_{L^2(\Omega)}\\& =  \langle v_0, \partial_{\bar z}\overline{\varphi}\rangle_{\mathcal{D}'(\R^2),\mathcal{D}(\R^2)}\\
	& = -  \langle \partial_{\bar z }v_0,\overline{\varphi}\rangle_{\mathcal{D}'(\R^2),\mathcal{D}(\R^2)}.
\end{align*}
Thus $\partial_{\bar z}v_0 = h_0 \in L^2(\R^2)$ and by Lemma \ref{lem:equivfullspace}, $v_0 \in H^1(\R^2)$ and $v\in H_0^1(\Omega)$. The proof for $u \in \dom(\partial_{\rm ah})$ is handled similarly.
\end{proof}

\section{Bergman and Hardy spaces on $\Omega$}\label{sec:berghard}
We introduce $\cA_{\rm h}^2(\Omega)$ and $\cA_{\rm ah}^2(\Omega)$ the holomorphic and anti-holomorphic Bergman spaces on $\Omega$, respectively. They are defined as
\[
	\cA_{\rm h}^2(\Omega) := \{u \in Hol(\Omega)\cap L^2(\Omega)\},\quad \cA_{\rm ah}^2(\Omega) := \{u : \overline{u}\in \cA_{\rm h}^2(\Omega)\},
\]
where $Hol(\Omega)$ denotes the space of holomorphic functions in $\Omega$. The holomorphic and anti-holomorphic Hardy spaces, denoted $\cH_{\rm h}^2(\Omega)$ and $\cH_{\rm ah}^2(\Omega)$, respectively, are defined as
\begin{equation}
	\cH_{\rm h}^2(\Omega) := \{u\in \cA_{\rm h}^2(\Omega) : \Gamma^+u \in L^2(\partial\Omega)\},\quad \cH_{\rm ah}^2(\Omega) := \{u : \overline{u} \in \cH_{\rm h}^2(\Omega)\}.
	\label{eqn:Hardyspace}
\end{equation}
This section aims to describe explicitely the Bergman and Hardy spaces on $\Omega$ in terms of Cauchy integrals and Szeg\"o projectors that we define now.

For $f \in C^\infty(\partial\Omega)$ consider the Cauchy integrals defined for $z\in \C \setminus \partial\Omega$ by
\[
	\Phi_{\rm h}(f) (z) := \frac{1}{2\rmi \pi}\int_{\partial\Omega}\frac{f(\xi)}{\xi - z}d\xi,\quad 	\Phi_{\rm ah}(f) (z) := -\frac{1}{2 \rmi \pi} \int_{\partial\Omega}\frac{f(\xi)}{\overline{\xi} - \overline{z}}d\overline{\xi}.
\]
It is well-known (see \cite[\S 4.1.2.]{SV}) that $\Phi_{\rm h}(f)$ (resp. $\Phi_{\rm ah}(f)$) defines a holomorphic function (resp. anti-holomorphic function) in $\R^2\setminus\partial\Omega$.

The well-known Plemelj-Sokhotski formula (see \cite[Thm. 4.1.1]{SV}) state that for $f \in C^\infty(\partial\Omega)$ the functions $\Phi_{\rm h}(f)$ and $\Phi_{\rm ah}(f)$ have an interior and an exterior Dirichlet trace, denoted respectively $\gamma_0^+$ and $\gamma_0^-$, such that:
\begin{equation}\label{eqn:Plemeljdef}
	\gamma_0^\pm \Phi_{\rm h}(f) =  \pm \frac12 f + \frac12 S_{\rm h}f,\quad \gamma_0^\pm \Phi_{\rm ah}(f) = \pm  \frac12 f + \frac12 S_{\rm ah}f.
\end{equation}
Let $\sharp \in \{{\rm h},{\rm ah}\}$, note that by \cite[Theorem 3.1.]{Bell}, for $f\in C^\infty(\partial\Omega)$ we know that $\Phi_\sharp(f)|_{\Omega} \in C^\infty(\overline{\Omega})$ as well as $\Phi_\sharp(f)|_{\R^2\setminus\overline{\Omega}}\in C^\infty(\R^2\setminus\Omega)$. In particular, the traces $\gamma_0^\pm\Phi_\sharp(f)$ coincide with $\Gamma^\pm \Phi_\sharp(f)$, where $\Gamma^\pm$ are the trace operators defined in Lemma \ref{lem:extop} and Remark \ref{rem:extop}.
\begin{defin}\label{def:szeproj} We define the Szeg\"o projectors in $C^\infty(\partial\Omega)$ by
\begin{equation}\label{eqn:Plemeljdef}
	\Pi_{\rm h}^\pm := \pm \Gamma^\pm \Phi_{\rm h},\quad \Pi_{\rm ah}^\pm := \pm\Gamma^\pm \Phi_{\rm ah}.
\end{equation}
\end{defin}

\begin{prop} Let $s\in\R$ and $\sharp \in \{{\rm h},{\rm ah}\}$. The Szeg\"o projectors $\Pi_{\sharp}^\pm$ extend uniquely into bounded linear operators from $H^s(\partial\Omega)$ onto itself. Moreover, $\Pi_\sharp^\pm$ are projectors and $\Pi_\sharp^+ + \Pi_\sharp^-= 1 $.

\label{prop:proj_conti}
\end{prop}
\begin{proof}Remark that for $\sharp\in\{{\rm h},{\rm ah}\}$ and $f \in C^\infty(\partial\Omega)$, there holds
\[
	\Pi_{\sharp}^\pm f = \frac12 f \pm \frac12 S_{\sharp}f.
\]
By Proposition \ref{prop:mapcauchy}, $\Pi_\sharp^\pm$ extends into a bounded linear operator from $H^s(\partial\Omega)$ onto itself for all $s\in\R$.

Let $s\in\mathbb{R}$ and $f\in H^s(\partial\Omega)$. A fundamental fact is that $S_{\rm h}^2f = f$ (see \cite[Eqn. (4.10)]{SV}), in particular it implies that $S_{\rm ah}^2f = f$. Hence, we obtain
\begin{align*}
	(\Pi_\sharp^\pm)^2 &= (\frac12 \pm \frac12S_{\sharp})(\frac12 \pm \frac12S_{\sharp})\\&=\frac14 + \frac14 S_{\sharp}^2 \pm \frac12 S_\sharp\\
	&= \frac12 \pm \frac12S_\sharp\\
	& = \Pi_\sharp^\pm
\end{align*}
Hence $\Pi_\sharp^\pm$ are projectors and one easily checks that $\Pi_\sharp^+ + \Pi_\sharp^- = 1$.
\end{proof}

The main goal of this section is to prove the following description of the Bergman and Hardy spaces. As we will see further on in Proposition \ref{prop:extphi}, this description relies on an extension of the operators $\Phi_\sharp$ to Sobolev spaces on the boundary $\partial\Omega$ ($\sharp \in \{{\rm h},{\rm ah}\}$).
\begin{thm}\label{thm:berhar} Let $\sharp \in \{{\rm h},{\rm ah}\}$. The Bergman spaces satisfy
\[
	\cA_\sharp^2(\Omega) = \{\Phi_\sharp(f) : f \in H^{-\frac12}(\partial\Omega), \Pi_\sharp^- f = 0\}.
\]
The Hardy spaces verify
\[
	\cH_\sharp^2(\Omega) = \{\Phi_\sharp(f) : f \in L^{2}(\partial\Omega), \Pi_\sharp^- f = 0\}.
\]
\end{thm}

\subsection{Potential theory of the Wirtinger derivatives}

In this paragraph we prove the following proposition.
\begin{prop}\label{prop:extphi} Let $\sharp \in \{{\rm h}, {\rm ah}\}$ and $s\in\{-\frac12,0,\frac12\}$. The operator $\Phi_{\sharp}$ extends uniquely into a bounded operator from $H^{s}(\partial\Omega)$ to $H^{s+\frac12}(\Omega)$ also denoted $\Phi_\sharp$.
\end{prop}
In order to prove Proposition \ref{prop:extphi}, we will need a few lemma. Let us start by defining fundamental solutions of the Wirtinger operators $\partial_{\rm h}$ and $\partial_{\rm ah}$:
\[
	\varphi_{\rm h}(x)  = \frac{1}{\pi (x_1 + i x_2)},\quad \varphi_{\rm ah}(x) = \frac{1}{ \pi(x_1 - i x_2)}.
\]
\begin{lem} Let $\sharp\in \{{\rm h}, {\rm ah}\}$. The linear map
\[
	N_\sharp : u\in L^2(\Omega) \mapsto \varphi_{\sharp}\ast u_0
\]
is bounded from $L^2(\Omega)$ to $H_{loc}^1(\R^2)$. Here $u_0$ denotes the extension of $u$ by zero to the whole $\R^2$.
\label{lem:newtonpot}
\end{lem}
\begin{proof} Let us prove it for $\sharp = {\rm h}$ the proof for $\sharp = {\rm ah}$ being similar. In the space of distributions $\mathcal{D}'(\R^2)$, there holds
\begin{equation}\label{eqn:fundsolz}
	\partial_{\bar z} \varphi_{\rm h} = \delta_0,
\end{equation}
where $\delta_0$ is the delta-Dirac distribution.

Now, for $u$ in the Schwartz space $\mathcal{S}(\R^2)$ recall that the Fourier transform of $u$ is defined as
\[
	\widehat{u}(k) :=  \int_{\R^2} f(x) e^{-2\rmi\pi \langle x,k\rangle_{\R^2}} dx,\quad \text{for all } k\in\R^2
\]
and $\widehat{u}\in \mathcal{S}(\R^2)$. The Fourier transform extends to the space of tempered distribution $\mathcal{S}'(\R^2)$ and as $\delta_0 \in \mathcal{S}'(\R^2)$, the Fourier transform of \eqref{eqn:fundsolz} yields
\[
	\widehat{\varphi_{\rm h}}(k) = \frac{1}{\pi \rmi (k_1 + \rmi k_2)},\quad k = (k_1,k_2)\in\R^2\setminus\{(0,0)\}.
\]
Let $K$ be a compact subset of $\R^2$ and take $u \in L^2(\Omega)$. We extend $u$ by zero to $\R^2$ and denote this extension $u_0 \in L^2(\R^2)$.
\begin{align*}
	\|\varphi_{\rm h}\ast u_0\|_{H^1(K)}^2 &\leq \|\varphi_{\rm h}\ast u_0\|_{L^2(K)}^2 + \int_{\R^2} |k|^2|(\widehat{\varphi_{\rm h}\ast u_0})(k)|^2 dk\\&= \|\varphi_{\rm h}\ast u_0\|_{L^2(K)}^2 + \int_{\R^2} |k|^2|\widehat{\varphi_{\rm h}}(k)\widehat{u_0}(k)|^2dk\\& = \|\varphi_{\rm h}\ast u_0\|_{L^2(K)}^2 + \frac{1}{\pi^2}\int_{\R^2} |\widehat{u_0}(k)|^2dk\\& = \|\varphi_{\rm h}\ast u_0\|_{L^2(K)}^2 + \frac{1}{\pi^2}\|u\|_{L^2(\Omega)}^2.
\end{align*}
Now, let $R>0$ be such that $K \subset \{ x \in \R^2 : |x| < R\}$ and $\overline{\Omega} \subset \{ x \in \R^2 : |x| < R\}$. Consider a cut-off function $\chi \in C_0^\infty([0,+\infty))$ such that
\[
	\chi(\rho) = 1\text{ whenever } 0\leq\rho < 2R,\quad \chi(\rho) = 0 \text{ whenever } \rho > 3R.
\]
Define the function $u_\chi$ as
\[
	u_\chi (x) := \int_{\R^2}\chi(|x-y|) \varphi_{\rm h}(x-y) u_0(y) dy.
\]
As defined, $u_\chi|_{K}\equiv (\varphi_{\rm h}\ast u_0)|_{K}$. Hence, we get
\begin{align*}
	\|\varphi_{\rm h}\ast u_0\|_{L^2(K)} = \|u_\chi\|_{L^2(K)} &\leq \|u_\chi\|_{L^2(\R^2)}\\&\leq \|\chi(|\cdot|)\varphi_{\rm h}\|_{L^1(\R^2)} \|u\|_{L^2(\Omega)},
\end{align*}
where we have used Young's inequality because $\chi(|\cdot|)\varphi_{\rm h} \in L^1(\R^2)$. Indeed, there holds
\[
	\|\chi(|\cdot|)\varphi_{\rm h}\|_{L^1(\R^2)} \leq \frac1\pi\int_{B(0,3R)} \frac{1}{|x|}dx = 6 R.
\]
In particular, there exists $c_K > 0$, such that
\[
	\|\varphi_{\rm h}\ast u_0\|_{H^1(K)} \leq c_K \|u\|_{L^2(\Omega)}.
\]
Hence, for any compact $K\subset \R^2$, $N_\sharp$ is a bounded linear operator from $L^2(\Omega)$ to $H^1(K)$ and the proposition is proved.
\end{proof}
Next, we recall that the Dirichlet trace on $\partial\Omega$ of a function in $H_{loc}^1(\R^2)$ can be defined as
\[
	\Gamma : H_{loc}^1(\R^2) \to H^{\frac12}(\partial\Omega)
\]
and is a bounded linear operator from $H_{loc}^1(\R^2)$ to $H^{\frac12}(\partial\Omega)$ (see \cite[Thm. 3.37]{McLean}).

Moreover, for $s \in [0,\ell]$, we introduce ${\bf t}(s) := \gamma_1'(s) + \rmi \gamma_2'(s)$ the expression of the tangent vector in the complex plane at the point $\gamma_1(s) + \rmi \gamma_2(s)$.
\begin{lem} The dual adjoints of $({\bf t} \Gamma N_{\rm h})$ and $(\overline{{\bf t}}\Gamma N_{\rm ah})$, denoted $({\bf t} \Gamma N_{\rm h})'$ and $(\overline{{\bf t}}\Gamma N_{\rm ah})'$ respectively, are bounded linear maps from $H^{-\frac12}(\partial\Omega)$ to $L^2(\Omega)$. Moreover if $f\in C^\infty(\partial\Omega)$, in $L^2(\Omega)$ there holds:
\[
\Phi_{\rm ah }(f) = \frac{\rmi}2({\bf t}\Gamma N_{\rm h})'(f),\quad \Phi_{\rm h}(f) = -\frac{\rmi}2(\overline{{\bf t}}\Gamma N_{\rm ah})'(f).
\]
\label{lem:map}
\end{lem}
\begin{proof} Thanks to Lemma \ref{lem:newtonpot} and the mapping properties of $\Gamma$ we know that $\Gamma N_\sharp$ is a bounded linear map from $L^2(\Omega)$ to $H^{\frac12}(\partial\Omega)$ (for $\sharp \in \{{\rm h}, {\rm ah}\}$). As $\Omega$ is smooth, ${\bf t} \in C^\infty(\partial\Omega)$ and $\overline{{\bf t}} \in C^\infty(\partial\Omega)$. In particular the multiplication operators by ${\bf t}$ and $\overline{{\bf t}}$ are bounded and invertible in $H^{\frac12}(\Omega)$. Hence, their dual adjoints satisfy the expected mapping property.

Now, pick $f\in C^\infty(\partial\Omega)$ and $v\in L^2(\Omega)$. Denoting by $v_0$ the extension of $v$ by zero to the whole $\R^2$ and using Fubini's theorem, there holds
\begin{align*}
	\langle ({\bf t}\Gamma N_{\rm h})'f, v\rangle_{L^2(\Omega)} &= \langle f, {\bf t} \Gamma N_{\rm h}v\rangle_{H^{-\frac12}(\partial\Omega),H^{\frac12}(\partial\Omega)}\\
	&= \langle f, {\bf t} \Gamma N_{\rm h}v\rangle_{L^2(\partial\Omega)}\\
	& = \int_{x\in\partial\Omega}\int_{y\in\R^2} \frac{f(x)\overline{v_0(y)}\overline{{\bf t}(x)}}{\pi\big((x_1-\rmi x_2) -(y_1 - \rmi y_2)\big)}dyds(x)\\
	& = \int_{y\in \R^2}  \frac1\pi\left(\int_{s=0}^\ell \frac{f(\gamma(s))(\gamma_1'(s) - \rmi \gamma_2'(s))}{(\gamma_1(s) - \rmi \gamma_2(s)) - (y_1 - \rmi y_2)}ds\right)\overline{v_0(y)} dy\\
	&=\int_{y\in \mathbb{R}^2} \Big(\frac1{\pi}\int_{\xi\in\partial \Omega} \frac{f(\xi)}{\overline{\xi} - (y_1 - \rmi y_2)}d\overline{\xi}\Big)\overline{v_0(y)} dy\\
	& = \langle - 2 \rmi \Phi_{\rm ah}(f),v\rangle_{L^2(\Omega)}.
\end{align*}
The proof for $(\overline{{\bf t}}\Gamma N_{\rm ah})'$ goes along the same lines, which concludes the proof of this lemma.
\end{proof}
For further use, we still denote $\Phi_{\rm ah}$ and $\Phi_{\rm h}$ the operators $\frac{\rmi}2({\bf t} \Gamma N_{\rm h})'$ and $-\frac{\rmi}2(\overline{{\bf t}}\Gamma N_{\rm ah})'$. Now, for $\sharp \in \{{\rm h},{\rm ah}\}$, when considering the operators
\[
	\Phi_{\sharp} : \big(C^\infty(\partial\Omega),\|\cdot\|_{H^{-\frac12}(\partial\Omega)}\big) \rightarrow \big(\dom(\partial_\sharp), \|\cdot\|_\sharp\big) 
\]
they are bounded operators because for any $f \in C^\infty(\partial\Omega)$, $\Phi_{\rm h}(f)$ and $\Phi_{\rm ah}(f)$ are holomorphic and anti-holomorphic in $\Omega$, respectively.
The density of $C^\infty(\partial\Omega)$ in $H^{-\frac12}(\partial\Omega)$ yields for each operator a unique extension to $H^{-\frac12}(\partial\Omega)$ which coincide with the previous one. In particular, for any $f \in H^{-\frac12}(\partial\Omega)$, $\Phi_\sharp(f) \in \dom(\partial_\sharp)$ and $\partial_\sharp \Phi_\sharp(f) = 0$.

Now, we have collected all the tools to prove Proposition \ref{prop:extphi}.
\begin{proof}[Proof of Proposition \ref{prop:extphi}] For $s = -\frac12$, Proposition \ref{prop:extphi} holds true, because of Lemma \ref{lem:map} and the density of $C^\infty(\partial\Omega)$ in $H^{-\frac12}(\partial\Omega)$.
Let us prove it for $s = \frac12$. Remark that $\Phi_\sharp(f) \in \dom(\partial_\sharp)$ so if $f\in H^{\frac12}(\partial\Omega)$ we also have $\Gamma^+\Phi_\sharp(f) = \Pi_\sharp^+ f \in H^{\frac12}(\partial\Omega)$ by Proposition \ref{prop:proj_conti}. Hence, by Lemma \ref{lem:ellipregul}, $\Phi_\sharp(f) \in H^1(\Omega)$.

Let us use the closed graph theorem and take a sequence of functions $f_n \in H^{\frac12}(\partial\Omega)$ such that $f_n \to f$ in the $H^{\frac12}(\partial\Omega)$-norm. Assume also that $\Phi_\sharp(f_n) \to u \in H^1(\Omega)$ where the convergence holds in the $H^1(\Omega)$-norm.

Because of the continuous embedding of $H^{\frac12}(\partial\Omega)$ into $H^{-\frac12}(\partial\Omega)$, $f_n \to f$ also in the $H^{-\frac12}(\partial\Omega)$-norm. In particular, by Proposition \ref{prop:extphi} for $s = -\frac12$, $\Phi_\sharp(f_n) \to \Phi_\sharp(f)$ in $L^2(\Omega)$. Consequently, the equality $u = \Phi_\sharp(f)$ holds not only in $L^2(\Omega)$ but also in $H^1(\Omega)$ and by the closed graph theorem, $\Phi_\sharp$ is a continuous linear map between $H^{\frac12}(\partial\Omega)$ and $H^1(\Omega)$.

The result for $s = 0$ holds by (real) interpolation theory (see \cite[Prop. 2.1.62. \& Prop. 2.3.11. \& Prop. 2.4.3.]{SauSch}).
\end{proof}

\subsection{Explicit description of the Bergman and Hardy spaces}
Let us prove Theorem \ref{thm:berhar}, starting with the following proposition concerning the Bergman spaces.
\begin{prop}\label{prop:propberg} Let $\sharp \in \{{\rm h},{\rm ah}\}$. There holds:
\begin{equation}\label{eqn:propberg}
	\cA_\sharp^2(\Omega) = 	\{ \Phi_\sharp (f) : f\in H^{-\frac12}(\partial\Omega) \text{ such that } \Pi_\sharp^- f = 0\},\quad \sharp\in\{{\rm h},{\rm ah}\}.
\end{equation}
Moreover, for all $f\in H^{-\frac12}(\partial\Omega)$ there holds
\[
	\Phi_\sharp(f) = \Phi_\sharp(\Pi_\sharp^+ f).
\]
\end{prop}
\begin{proof} Denote $\mathcal{E}_\sharp$ the set on the right-hand side of \eqref{eqn:propberg}. We prove it for $\sharp = {\rm h}$, the proof for $\sharp = {\rm ah}$ being similar.
\paragraph{\underline{Inclusion $\mathcal{E}_{\rm h} \subset \cA_{\rm h}^2(\Omega)$}}
Let $ u = \Phi_{\rm h}(f) \in \mathcal{E}_{\rm h}$, with $f \in H^{-\frac12}(\partial\Omega)$ such that $\Pi_{\rm h}^- f = 0$. By Proposition \ref{prop:extphi}, $\Phi_{\rm h}$ maps $H^{-\frac12}(\partial\Omega)$ to $L^2(\Omega)$ thus $u \in L^2(\Omega)$. Moreover, there holds $\partial_{\bar z}u = 0$ which implies that $u \in \mathcal{A}_{\rm h}^2 (\Omega)$.

\paragraph{\underline{Inclusion $\cA_{\rm h}^2(\Omega) \subset \mathcal{E}_{\rm h}$}} For $u \in C^\infty(\overline{\Omega})$, $x \in \Omega$ and $\varepsilon > 0$ sufficiently small there holds
\begin{align*}
	0 &= \frac1\pi\int_{\Omega\setminus B(x,\varepsilon)} \partial_{\bar z}\Big(\frac1{ (x_1 +\rmi x_2) - (y_1 +\rmi y_2)}\Big) u(y) dy\\& = -\frac1\pi\int_{\Omega\setminus B(x,\varepsilon)} \frac{\partial_{\bar z} u(y)}{(x_1 +\rmi x_2) - (y_1 +\rmi y_2)}dy\\& \quad \quad+ \frac1{2\pi}\int_{\partial\Omega}\frac{u(y)}{(x_1+\rmi x_2) - (y_1+\rmi y_2)} \textbf{n}(y) ds(y)\\&\quad\quad\quad + \frac1{2\pi}\int_{\partial B(x,\varepsilon)}\frac{u(y)}{(x_1+\rmi x_2) - (y_1+\rmi y_2)} \frac{(x_1 +\rmi x_2) - (y_1 +\rmi y_2)}{|y-x|} ds(y)\\
	& := -A + B + C.
\end{align*}
However, we have
\[
	C = \frac1{2\pi}\int_{0}^{2\pi}u(x + \varepsilon(\cos t,\sin t)) d t \longrightarrow u(x),\quad \text{when }\varepsilon \to 0.
\]
By definition, if $\gamma : [0,\ell] \to \partial\Omega$ is a smooth arc-length parametrization of $\partial\Omega$ there holds
\begin{align*}
	B & = -\frac{\rmi}{2\pi} \int_{0}^\ell \frac{u(\gamma(t))}{(x_1+\rmi x_2) - (\gamma_1(t) + \rmi \gamma_2(t))}(\gamma_1'(t) + \rmi \gamma_2'(t)) dt\\
	& = -\frac{1}{2\rmi \pi} \int_{\partial\Omega} \frac{u(\xi)}{\xi- (x_1 + \rmi x_2)} d\xi =  - \Phi_{\rm h}(\Gamma^+ u)(x).
\end{align*}
In particular, we obtain
\begin{align}
	\nonumber\Phi_{\rm h}(\Gamma^+ u)(x) =\ & \frac1{2\pi}\int_{0}^{2\pi} u(x+\varepsilon (\cos(t),\sin(t))) dt\\& \nonumber\qquad\qquad\qquad - \frac1{\pi}\int_{\Omega\setminus B(x,\varepsilon)} \frac{\partial_{\bar z} u(y)}{(x_1 + \rmi x_2) - (y_1 + \rmi y_2)} dy\\ \nonumber=\ &\frac1{2\pi}\int_{0}^{2\pi} u(x+\varepsilon (\cos(t),\sin(t))) dt\\\label{eqn:decdist}&\qquad\qquad\qquad - \frac1{\pi}\int_{\R^2} \frac{\mathds{1}_{\R^2\setminus B(0,\varepsilon)}(x-y)}{(x_1 + \rmi x_2) - (y_1 + \rmi y_2)} (\partial_{\bar z} u(y)\mathds{1}_{\Omega}(y)) dy.
\end{align}
Note that the linear form on $C_0^\infty(\R^2)$ defined by
\[
	p.v.\Big(\frac1{x_1+\rmi x_2}\Big) := \varphi \in C_0^\infty(\R^2) \mapsto   \lim_{\varepsilon \to 0} \int_{\R^2}\frac{\mathds{1}_{\R^2 \setminus B(0,\varepsilon)}(x)}{x_1 +\rmi x_2}\varphi(x) dx \in \C
\]
belongs to $\cD'(\R^2)$. Remark that $(\partial_{\bar z}u \mathds{1}_{\Omega} )\in \cD'(\R^2)$ and has compact support. Hence, $p.v.(\frac1{x_1+\rmi x_2}) * (\partial_{\bar z}u \mathds{1}_{\Omega}) \in \cD'(\R^2)$ and taking the duality pairing with $\varphi \in C_0^\infty(\Omega)$ in \eqref{eqn:decdist} and $\varepsilon \to 0$ we get
\begin{equation}
	\langle\Phi_{\rm h}(\Gamma^+ u ) - u,\varphi\rangle_{\cD'(\Omega),\cD(\Omega)} = \frac1{\pi}\langle p.v.\Big(\frac1{x_1 +\rmi x_2}\Big)*(\partial_{\bar z} u \mathds{1}_{\Omega}),\varphi\rangle_{\cD'(\R^2),\cD(\R^2)}.
	\label{eqn:cvddistrib}
\end{equation}
Now, remark that $\cA_{\rm h}^2(\Omega) \subset \dom(\partial_{\rm h})$ and pick a sequence of $C^\infty(\overline{\Omega})$  functions $(v_n)_{n\in\N}$ which converges to $v \in \cA_{\rm h}^2(\Omega)$ in the norm of $\dom(\partial_{\rm h})$ when $n\to +\infty$. In particular, $(v_n)_{n\in\N}$ converges to $v$ and $(\partial_{\bar z} v_n)\mathds{1}_\Omega$ converges to $0$ when $n\to +\infty$ in $\cD'(\R^2)$. Using \eqref{eqn:cvddistrib} for $u = v_n$ and letting $n \to +\infty$ we obtain that in $\cD'(\Omega)$ there holds $v = \Phi_{\rm h}(\Gamma^+v)$ where we have used the continuity of the map $\Phi_{\rm h} \circ \Gamma^+ : \dom(\partial_{\rm h}) \to L^2(\Omega)$, and the continuity of the convolution in $\cD'(\R^2)$.
Now, remark that we also have $v = \Phi_{\rm h}(\Gamma^+ v)$ in $\cA_{\rm h}^2(\Omega)$ and taking the trace $\Gamma^+$ on both side of this identity we get
\[
	\Pi_{\rm h}^+\Gamma^+ v = \Gamma^+ v
\]
which implies $v = \Phi_{\rm h}(\Pi_{\rm h}^+ \Gamma^+ v)$ and proves the other inclusion.

\end{proof}
We are now in a good position to prove Theorem \ref{thm:berhar}.

\begin{proof}[Proof of Theorem \ref{thm:berhar}] Proposition \ref{prop:propberg} is precisely the first statement of Theorem \ref{thm:berhar} thus, the only thing left to prove is the statement for the Hardy spaces. Now, recall that for $\sharp\in\{{\rm h},{\rm ah}\}$, we have defined the Hardy spaces in \eqref{eqn:Hardyspace} and that we want to prove
\[
	\mathcal{H}_\sharp^2(\Omega) = \{\Phi_\sharp(f) : f \in L^2(\partial\Omega), \Pi_\sharp^-f = 0\}.
\]
Let $\mathcal{E}_\sharp$ be the set on the right-hand side, we prove both inclusions.

\underline{Inclusion $\mathcal{E}_\sharp\subset \mathcal{H}_\sharp^2(\Omega)$.}
Let $ u = \Phi_\sharp(f) \in \mathcal{E}_\sharp$, by definition $u \in \mathcal{A}_\sharp^2$ et $\Gamma^+ u = f \in L^2(\partial\Omega)\subset H^{-\frac12}(\partial\Omega)$ which proves this inclusion.

\underline{Inclusion $\mathcal{H}_\sharp^2(\Omega) \subset \mathcal{E}_\sharp$.}
Let $u \in \mathcal{H}_\sharp^2(\Omega)$. We know that in particular $u = \Phi_\sharp(f)$ for some $f \in H^{-\frac12}(\partial\Omega)$ such that $\Pi_\sharp^- f = 0$. But we have $\Gamma^+ u = f \in L^2(\partial\Omega)$ which proves this inclusion and concludes the proof.
\end{proof}
\subsection{Explicit description of the domain of the maximal Wirtinger operators}
In this paragraph, we prove the following descrition of the domains of the maximal Wirtinger operators introduced in Section \ref{sec:maxc-rop}. This description involves the Bergman spaces introduced in the beginning of Section \ref{sec:berghard}.
\begin{prop} Let $\sharp \in \{{\rm h}, {\rm ah}\}$. The following direct sum decomposition holds:
\[
	\dom(\partial_\sharp) =  \{u \in H^{1}(\Omega) : \Pi_\sharp^+ \Gamma^+ u = 0\} \dotplus \cA_\sharp^2(\Omega).
\]
\label{prop:dirsum}
\end{prop}

For $\sharp \in \{{\rm h},{\rm ah}\}$, the range of the trace operator $\Gamma^+ : \dom{(\partial_\sharp)} \to H^{-\frac12}(\partial\Omega)$ is of crucial importance to prove Proposition \ref{prop:dirsum}. We describe its range now, thanks to the Szeg\"o projectors introduced in \eqref{eqn:Plemeljdef} but first, we prove a regularization result.

\begin{lem}\label{lem:regular} Let $\sharp \in \{\rm h, {\rm ah}\}$. The operator $\Pi_\sharp^- \circ \Gamma^+$ is a bounded linear operator from $\dom(\partial_\sharp)$ to $H^{\frac12}(\partial\Omega)$.
\end{lem}
\begin{proof} Let $u\in \dom(\partial_{\rm h})$ and $u_n \in C^\infty(\overline{\Omega})$ be a sequence converging to $u$ in the $\|\cdot\|_{\rm h}$-norm when $n\to +\infty$. Pick $f\in C^\infty(\partial\Omega)$, an integration by parts yields:
\[
	\langle\Gamma^+u_n, \overline{\bf n}\  \Pi_{\rm ah}^+f\rangle_{L^2(\partial\Omega)} = 2 \langle\partial_{\bar z} u_n, \Phi_{\rm ah}(f)\rangle_{L^2(\Omega)}.
\]
It gives
\[
	|\langle\Gamma^+u_n, \overline{\bf n}\  \Pi_{\rm ah}^+f\rangle_{L^2(\partial\Omega)}| \leq 2 c \|u_n\|_{\rm h}\|f\|_{H^{-\frac12}(\partial\Omega)},
\]
for some $c>0$, where we have used Lemma \ref{lem:extop} and Proposition \ref{prop:extphi}. As in $L^2(\partial\Omega)$ there holds $S_{\rm ah}^* = - S_{\rm h}$ we get
\[
	(\overline{\mathbf{n}}\Pi_{\rm ah}^+)^* = \Pi_{\rm h}^-\mathbf{n}.
\]
In particular, there holds
\[
	|\langle\Gamma^+u_n, \overline{\bf n}\  \Pi_{\rm ah}^+f\rangle_{L^2(\partial\Omega)}| = |\langle  (\Pi_{\rm h}^-{\bf n}\Gamma^+u_n, f\rangle_{L^2(\partial\Omega)}| \leq 2 c \|u_n\|_{\rm h}\|f\|_{H^{-\frac12}(\partial\Omega)}.
\]
Letting $n\to +\infty$, we get $\Pi_{\rm h}^- {\bf n}\Gamma^+u \in H^{\frac12}(\partial\Omega)$ and that $\Pi_{\rm h}^-\circ H_{\bf n}\circ\Gamma^+$ is a linear bounded map from $H^{-\frac12}(\partial\Omega)$ to $H^{\frac12}(\partial\Omega)$. However, there holds
\[
	\Pi_{\rm h}^- \Gamma^+ u = \Big(\overline{\bf n} \Pi_{\rm h}^- {\bf n} - \overline{\bf n} [\Pi_{\rm h}^-,{\bf n}]\Big)\Gamma^+ u = \overline{\bf n} \Pi_{\rm h}^- {\bf n}\Gamma^+ u + \overline{\bf n} [S_{\rm h},{\bf n}]\Gamma^+ u.
\]
By \eqref{itm:2pseudo} Proposition \ref{prop:pseudocommut}, $[S_{\rm h},{\bf n}]\in \Psi_{\partial\Omega}^{-1}$ hence, it is a bounded operator from $H^{-\frac12}(\partial\Omega)$ to $H^{\frac12}(\partial\Omega)$. Finally, as the multiplication operator by $\overline{\bf n}$ is bounded in $H^{\frac12}(\partial\Omega)$ we obtain the expected result.

The case $u \in \dom(\partial_{\rm ah})$ is handled similarly.
\end{proof}

We are now in a good position to describe the range of the trace operator $\Gamma^+$.
\begin{corol} Let $\sharp \in \{{\rm h},{\rm ah}\}$. There holds
\[
	\ran (\Gamma^+) = \{f \in H^{-\frac12}(\partial\Omega) : \Pi_{\sharp}^- f \in H^{\frac12}(\partial\Omega) \}.
\]
\end{corol}

\begin{proof} Let us start by proving the reverse inclusion. Let $f$ be in the set on right-hand side, there holds $f = \Pi_\sharp^+ f + \Pi_\sharp^-f$. We know that there exists an extension operator $E^+$ from $H^{\frac12}(\partial\Omega)$ to $H^{1}(\Omega)$ such that $\Gamma^+ E^+ g = g$ for all $g \in H^{\frac12}(\partial\Omega)$. Now, if $\Pi_\sharp^- f \in H^{\frac12}(\partial\Omega)$, we set
\[
	u := \Phi_\sharp(\Pi_\sharp^+ f) + E^+(\Pi_\sharp^- f).
\]
It is easily seen that $u \in \dom(\partial_\sharp)$ and $\Gamma^+ u = \Pi_\sharp^+ f + \Pi_\sharp^- f = f$.

Now, let us prove the direct inclusion and pick $f \in \ran(\Gamma^+)$. We know that there exists $u\in \dom(\partial_\sharp)$ such that $f = \Gamma^+ u $. In particular, by Lemma \ref{lem:regular} we know that $\Pi_\sharp^- f = \Pi_\sharp^- \Gamma^+ u\in H^{\frac12}(\partial\Omega)$ which concludes the proof.

\end{proof}

We are now able to prove Proposition \ref{prop:dirsum}.

\begin{proof}[Proof of Proposition \ref{prop:dirsum}] First, let us prove that the sum is direct. Let $v = \Phi_\sharp(f) = u$ with $\Pi_\sharp^- f = 0$ and $\Pi_\sharp^+ \Gamma^+ u = 0$. Then, taking the traces we obtain:
\[
	\Gamma^+ v = \Pi_\sharp^+ f = \Pi_\sharp^- \Gamma^+ u,
\]
which implies $f = \Gamma^+ u = 0$. Consequently, $v = \Phi_\sharp(f) = 0$.

Second, let us pick $v \in \dom(\partial_\sharp)$. There holds
\[
	v = \Phi_\sharp(\Pi_\sharp^+\Gamma^+ v) + v - \Phi_\sharp(\Pi_\sharp^+\Gamma^+ v).
\]
However, remark that $u := v - \Phi_\sharp(\Pi_\sharp^+\Gamma^+ v) \in \dom(\partial_\sharp)$ and satisfies $\Gamma^+ u = \Pi_\sharp^- \Gamma^+ v \in H^{\frac12}(\partial\Omega)$ by Lemma \ref{lem:regular}. Hence, by Lemma \ref{lem:ellipregul}, we obtain $u \in H^1(\Omega)$ and $\Gamma^+ u \in \ker\Pi_\sharp^+ = \ran \Pi_\sharp^-$, which concludes the proof.
\end{proof}
\section{Variational characterization of the principal eigenvalue}\label{sec:vcproof}
The aim of this section is to prove Theorem \ref{thm:vf}. In \S \ref{subsec:para1} we describe precisely the domains $\dom(q_E^\Omega)$ and $\dom(H_E^\Omega)$, where $H_E^\Omega$ is the unique self-adjoint operator associated with $q_E^\Omega$ \emph{via} Kato's first representation theorem. In \S \ref{subsec:para2}, we investigate the behavior of the map $E \in [0,+\infty) \mapsto \mu_\Omega(E)$. Finally, in \S \ref{subsec:para3}, we prove Theorem \ref{thm:vf}.

\subsection{The quadratic form $q_E^\Omega$ and its associated self-adjoint operator $H_E^\Omega$}\label{subsec:para1}
For $E>0$, recall that $q_E^\Omega$ is defined in \eqref{eqn:deffq} on the domain consisting of the closure of the $C^\infty(\overline{\Omega})$ functions with respect to the norm of the quadratic form
\[
	N_E^\Omega(u) := \sqrt{\|\partial_{\bar z} u\|_{L^2(\Omega)}^2 + \|u\|_{L^2(\Omega)}^2 + E\|u\|_{L^2(\partial\Omega)}^2}.
\]
Remark that as defined, $q_E^\Omega$ is a closed, densely defined and bounded below quadratic form thus, by Kato's first representation theorem (see \cite[Chap. VI, Thm. 2.1]{Kat}), $q_E^\Omega$ is associated with a unique self-adjoint operator $H_E^\Omega$ acting in $L^2(\Omega)$ satisfying
\[
	\dom(H_E^\Omega)\subset\dom(q_E^\Omega).
\]
In this paragraph, we describe properties of the domains $\dom(q_E^\Omega)$ and $\dom(H_E^\Omega)$ and start with the domain of the quadratic form $q_E^\Omega$.

\begin{prop}\label{prop:descformdom} Let $E >0$. The form domain $\dom(q_E^\Omega)$ admits the following direct sum decomposition
\[
	\dom(q_E^\Omega) = \{u \in H^1(\Omega) : \Pi_{\rm h}^+ \Gamma^+ u = 0\} \dotplus \cH_{\rm h}^2(\Omega).
\]
Moreover, $\dom(q_E^\Omega)$ is continuously embedded in $H^{\frac12}(\Omega)$.
\end{prop}
\begin{proof} Set $\mathcal{E} = \{u \in H^1(\Omega) : \Pi_{\rm h}^+ \Gamma^+ u = 0\} \dotplus \cH_{\rm h}^2(\Omega)$ and remark that the sum is direct by the same arguments as in the proof of Proposition \ref{prop:dirsum}. We prove the set equality by proving both inclusions.
\paragraph{\underline{Inclusion $\mathcal{E}\subset\dom(q_{E}^\Omega)$}} Let $ v:= u + \Phi_{\rm h}(f) \in \mathcal{E}$ and take $(u_n)_{n\in \N}$ and $(f_n)_{n\in \N}$ two sequences of functions such that
\[
	\text{for all } n\in \N\quad u_n \in C^\infty(\overline{\Omega}),\quad  f_n \in C^\infty(\partial\Omega);
\]
and
\[
	\text{when } n\to +\infty\text{ there holds }\|u_n - u\|_{H^1(\Omega)} \to 0,\quad\|f_n - f\|_{L^2(\partial\Omega)}\to 0.
\]
By \cite[Theorem 3.1.]{Bell}, we have $v_n := u_n + \Phi_{\rm h}(f_n) \in C^\infty(\overline{\Omega})$ and for $E>0$, there exists $C>0$ such that there holds
\begin{align*}
	q_E^\Omega(v-v_n) +(E^2+1)\|v-v_n\|_{L^2(\Omega)}^2 =&\ \|\partial_{\bar z} (u - u_n)\|_{L^2(\Omega)}^2 \\& \quad+ E \|\Gamma^+(u - u_n) + \Pi_{\rm h}^+(f - f_n)\|_{L^2(\partial\Omega)}^2 \\&\quad\quad + \|(u-u_n) + (\Phi_{\rm h}(f-f_n))\|_{L^2(\Omega)}^2\\& \leq C\bigg(\|u-u_n\|_{H^1(\Omega)} + \|f- f_n\|_{L^2(\partial\Omega)}\bigg),
\end{align*}
where we have used the mapping properties of $\Phi_h$, $\Gamma^+$, $\Pi_{\rm h}^+$ and the continuity of the embedding of $L^2(\partial\Omega)$ into $H^{-\frac12}(\partial\Omega)$.
Letting $n\to +\infty$, we obtain that $v \in \dom(q_{E}^\Omega)$ and this inclusion is proved.
\paragraph{\underline{Inclusion $\dom(q_{E}^\Omega)\subset\mathcal{E}$}} For all $u \in C^\infty(\overline{\Omega})$, there holds
\[
	q_E^\Omega(u) + (E^2+1) \|u\|_{L^2(\Omega)}^2 \geq \|u\|_{\rm h}^2.
\]
In particular, the closure of $C^\infty(\overline{\Omega})$ for the norm $N_E^\Omega$ is included in $\dom(\partial_{\rm h})$. It rewrites $\dom(q_E^\Omega)\subset \dom(\partial_{\rm h})$ and by Proposition \ref{prop:dirsum}, any $v \in \dom(q_E^\Omega)$ writes $v = u + \Phi_{\rm h}(f)$, for some $u \in H^1(\Omega)$ with $\Pi_{\rm h}^+ \Gamma^+ u = 0$ and some $f \in H^{-\frac12}(\partial\Omega)$ with $\Pi_{\rm h}^- f = 0$. Now, if $v_n \in C^\infty(\overline{\Omega})$ converges to $v\in \dom(q_E^\Omega)$ in the norm of the quadratic form, we have
\[
	\|\Gamma^+ v - \Gamma^+ v_n\|_{L^2(\partial\Omega)} \leq E^{-1} q_E^\Omega(v-v_n)\to 0,\quad n\to +\infty.
\]
In particular $\Gamma^+ v  = \Gamma^+ u + f \in L^2(\partial\Omega)$ and as $\Gamma^+ u \in H^{\frac12}(\partial\Omega)$ we get $f\in L^2(\partial\Omega)$ which concludes the proof of this inclusion.

Let us consider the inclusion map
\[
	\mathcal{I} :=  \dom(q_E^\Omega) \to H^{\frac12}(\Omega),\quad (\mathcal{I} u) = u.
\]
By Proposition \ref{prop:extphi} for $s = 0$, this map is well-defined. Consider $v_n := u_n + \Phi_{\rm h}(f_n) \in \dom(q_E^\Omega)$ which converges to $v$ in the norm of the quadratic form $q_E^\Omega$ and assume that $v_n \to w$ in the $H^\frac12(\Omega)$-norm. In particular, as $v \in \dom(q_E^\Omega)$, there holds $v = u + \Phi_{\rm h}(f)$ for some $u\in H^1(\Omega)$ and $f\in H^{\frac12}(\partial\Omega)$ as in the definition of $\mathcal{E}$. In particular, in $\cD'(\Omega)$ we obtain
\[
	u + \Phi_{\rm h}(f) = w
\]
and as both terms belong to $H^{\frac12}(\Omega)$, the closed graph theorem gives that $\mathcal{I}$ is continuous.
\end{proof}
Because of the compact embedding of $H^{\frac12}(\Omega)$ into $L^2(\Omega)$, an immediate corollary of Proposition \ref{prop:descformdom} reads as follows.
\begin{corol}Let $E>0$, the operator $H_E^\Omega$ has compact resolvent and its spectrum consists of a non-decreasing sequence of eigengalues denoted $(\mu_j^\Omega(E))_{j\geq1}$. Moreover, there holds
\[
	\mu_j^\Omega(E) =\inf_{\tiny{\begin{array}{c}F \subset \dom{(q_E^\Omega)}\\ \dim F = j\end{array}}}\sup_{u \in F\setminus\{0\}} \frac{4 \int_\Omega |\partial_{\bar z} u|^2 dx - E^2 \int_{\Omega}|u|^2dx + E \int_{\partial\Omega} |u|^2 ds}{\int_\Omega |u|^2 dx}.
\]
\end{corol}
\begin{rem} For $E = 0$, the counterpart of Propostion \ref{prop:descformdom},  would read
\[
	\dom(q_0^\Omega) = \{u \in H^1(\Omega) : \Pi_{\rm h}^+ \Gamma^+ u = 0\} \dotplus \cA_{\rm h}^2(\Omega).
\]
In particular, note that $\dom(q_0^\Omega)$ can not be included in any Sobolev space $H^{s}(\Omega)$, ($s>0$). Indeed, for any Bergman function $u \in \cA_{\rm h}^2(\Omega)$, there holds $q_0^\Omega(u) = 0$ which implies that for all $j \geq 1$ we have $\mu_j^\Omega(0) = 0$. Thus $0$ is an eigenvalue of $H_0^\Omega$ of infinite multiplicity which would not be possible if we had $\dom(q_0^\Omega) \subset H^s(\Omega)$ because of the compact embedding of $H^s(\Omega)$ in $L^2(\Omega)$. This phenomena is reminiscent of what happens for the Dirac operator with zig-zag boundary conditions as discussed in \cite{Schm95}.

\end{rem}
We conclude this paragraph by a description of the domain of the operator $H_E^\Omega$.
\begin{prop} Let $E>0$, there holds:
\[
	\dom(H_E^\Omega) = \{u \in H^1(\Omega) : \partial_{\bar z} u \in H^1(\Omega)\text{ and } \partial_{\bar z}u + {\bf n}\frac{E}2 u = 0 \text{ on }\partial\Omega\}.
\]
\label{prop:domope}
\end{prop}
\begin{proof} Let $\mathcal{E}$ denote the set in the right-hand side of Proposition \ref{prop:domope}. The proof is performed proving both inclusions.
\paragraph{\underline{Inclusion $\dom(H_E^\Omega)\subset\mathcal{E}$}}

Let $u\in \dom(H_E^\Omega)$ and $v \in C_0^\infty(\Omega)$, there holds
\begin{align*}
	\langle H_E^\Omega u, \overline{v}\rangle_{\cD'(\Omega),\cD(\Omega)} = \langle H_E^\Omega u, v\rangle_{L^2(\Omega)} &= q_E^\Omega[u,v]\\
					& = 4 \langle\partial_{\bar z}u,\partial_{z} \overline{v}\rangle_{\cD'(\Omega),\cD(\Omega)} - E^2\langle u,\overline{v}\rangle_{\cD'(\Omega),\cD(\Omega)}\\
					& = \langle (-\Delta -E^2)u, \overline{v}\rangle_{\cD'(\Omega),\cD(\Omega)},
\end{align*}
where $q_E^\Omega[\cdot,\cdot]$ denotes the sesquilinear form associated with the quadratic form $q_E^\Omega$. Hence, in $L^2(\Omega)$, there holds $H_E^\Omega u = (-\Delta - E^2) u$. Remark that if $u \in \dom(H_E^\Omega)$ then $\partial_{\bar z} u \in \dom(\partial_{\rm ah})$, in particular, by Green's Formula \eqref{eqn:Green}, for all $v \in C^\infty(\overline{\Omega})$ we get:
\begin{align*}
\langle H_E^\Omega u, v\rangle_{L^2(\Omega)} &=\ -4 \langle \partial_z (\partial_{\bar z}u), v\rangle_{L^2(\Omega)} - E^2\langle u,v\rangle_{L^2(\Omega)}\\
& = \ 4\langle \partial_{\bar z} u, \partial_{\bar z}v \rangle_{L^2(\Omega)} - E^2\langle u,v\rangle_{L^2(\Omega)} - 2\langle\overline{\bf n} \Gamma^+\partial_{\bar z}u,\Gamma^+v\rangle_{H^{-\frac12}(\partial\Omega),H^{\frac12}(\partial\Omega)}\\
& = \ q_E^\Omega[u,v] - \langle2\overline{\bf n} \Gamma^+\partial_{\bar z}u + E u,\Gamma^+v\rangle_{H^{-\frac12}(\partial\Omega),H^{\frac12}(\partial\Omega)}.
\end{align*}
As $v \in \dom(q_E^\Omega)$ we necessarily have $\langle2\overline{\bf n} \Gamma^+\partial_{\bar z}u + E u,\Gamma^+v\rangle_{H^{-\frac12}(\partial\Omega),H^{\frac12}(\partial\Omega)} = 0$. As this is true for all $v \in C^\infty(\overline{\Omega})$ we obtain
\begin{equation}\label{eqn:bc}
	2\overline{\bf n} \Gamma^+\partial_{\bar z}u + E \Gamma^+u = 0,\quad \text{in } H^{-\frac12}(\partial\Omega).
\end{equation}
Taking the Szeg\"o projectors in \eqref{eqn:bc} we obtain
\[
	(\Gamma^+\partial_{\bar z} u) + \frac{\bf n}{2}E \Gamma^+u = 0 \Longleftrightarrow \left\{\begin{array}{lcl}\Pi_{\rm ah}^+(\Gamma^+(\partial_{\bar z} u)) + \frac{E}2 \Pi_{\rm ah}^+{\bf n}\Gamma^+ u &=& 0\\
	\Pi_{\rm ah}^-(\Gamma^+(\partial_{\bar z} u)) + \frac{E}2 \Pi_{\rm ah}^-{\bf n}\Gamma^+ u &=& 0\end{array}\right.
\]
Nevertheless, there holds
\[
	\Pi_{\rm ah}^-  = \Pi_{\rm h}^+ - \frac12(S_{\rm h} + S_{\rm ah}),\quad \Pi_{\rm ah}^+ = \Pi_{\rm h}^- + \frac12(S_{\rm ah} + S_{\rm h}).
\]
In particular, we get
\begin{align*}
	\Pi_{\rm ah}^- (\Gamma^+(\partial_{\bar z} u)) = - \frac{E}2 \Pi_{\rm ah}^-(\mathbf{n} \Gamma^+u) &= -\frac{E}2\Big(\mathbf{n}\Pi_{\rm h}^+\Gamma^+ u + [\Pi_{\rm h}^+,\mathbf{n}]\Gamma^+u - \frac12(S_{\rm h} + S_{\rm ah})(\mathbf{n}\Gamma^+u)\Big)\\ &= -\frac{E}2\Big(\mathbf{n}\Pi_{\rm h}^+\Gamma^+u + [S_{\rm h},\mathbf{n}]\Gamma^+u - \frac12(S_{\rm h} + S_{\rm ah})(\mathbf{n}\Gamma^+u)\Big).
\end{align*}
It rewrites
\[
	\Pi_{\rm h}^+ \Gamma^+ u = -\overline{\bf n} \Big(\frac2E  \Pi_{\rm ah}^- \Gamma^+(\partial_{\bar z} u) + [S_{\rm h},{\bf n}]\Gamma^+u - \frac12(S_{\rm h}+S_{\rm ah})({\bf n}\Gamma^+ u)\Big).
\]
Remark that the right-hand side belongs to $H^{\frac12}(\partial\Omega)$. This holds for the first term because of Lemma \ref{lem:regular} and for the last two-terms because of Proposition \ref{prop:pseudocommut}. As $\Pi_{\rm h}^- \Gamma^+ u \in H^{\frac12}(\partial\Omega)$ by Lemma \ref{lem:regular}, we get $\Gamma^+ u = \Pi_{\rm h}^+ \Gamma^+ u + \Pi_{\rm h}^- \Gamma^+ u \in H^{\frac12}(\partial\Omega)$ thus, by Lemma \ref{lem:ellipregul}, $u \in H^1(\Omega)$. In particular $\Pi_{\rm ah}^+(\Gamma^+(\partial_{\bar z} u)) = - \frac{E}2\Pi_{\rm ah}^+{\bf n}\Gamma^+ u \in H^{\frac12}(\partial\Omega)$ and as $\Pi_{\rm ah}^- \Gamma^+ (\partial_{\bar z}u) \in H^{\frac12}(\partial\Omega)$ by Lemma \ref{lem:regular} we obtain $\Gamma^+ \partial_{\bar z} u = \Pi_{\rm ah}^- \Gamma^+ (\partial_{\bar z}u) + \Pi_{\rm ah}^+ \Gamma^+ (\partial_{\bar z}u) \in H^{\frac12}(\partial\Omega)$ and by Lemma \ref{lem:ellipregul} we obtain $\partial_{\bar z}u \in H^1(\Omega)$. It concludes the proof of this inclusion.
\paragraph{\underline{Inclusion $\mathcal{E}\subset \dom(H_E^\Omega)$}} Pick $u\in \mathcal{E}$. One easily sees that $(-\Delta - E^2)u \in L^2(\Omega)$, moreover for all $v \in \dom(q_E^\Omega)$, there holds
\[
	q_E^\Omega[u,v] = \langle(-\Delta - E^2)u,v\rangle_{L^2(\Omega)}.
\]
By definition of $H_E^\Omega$ it implies $u \in \dom(H_E^\Omega)$ and $H_E^\Omega u = (-\Delta -E^2)u$. 
\end{proof}

\subsection{Concavity of the first min-max level}\label{subsec:para2}
In this paragraph we investigate the behavior of the first min-max level $\mu^\Omega(E)$ with respect to the spectral parameter $E>0$. This behavior is illustrated in Figure \ref{fig:figure4} for various domains $\Omega$.
\begin{prop}\label{prop:monotonicity} The map $\mu^\Omega: E \geq 0 \mapsto \mu^\Omega(E)$ verifies the following properties.
\begin{enumerate}
\item\label{itm:1} $\mu^\Omega$ is a continuous and concave function on $\R_+$.
\item\label{itm:2} We have $\mu^\Omega(0) = 0$ and there exists $E_\star^\Omega > 0$ such that for all $E \in (0,E_\star^\Omega)$ there holds $\mu^\Omega(E) > 0$.
\item\label{itm:4} Let $0<E_1<E_2$, there holds
\[
	\mu^\Omega(E_2) \leq \frac{E_2}{E_1}\mu^\Omega(E_1) - E_2(E_2-E_1)
\]
In particular, if $\mu^\Omega(E_1) = 0$ (resp. $\mu^\Omega(E_2) = 0$) there holds $\mu^\Omega(E_2) < 0$ (resp. $\mu^\Omega(E_1)>0$).
\end{enumerate}
\end{prop}

\begin{proof} As for all $u\in \dom(q_E^\Omega)$ the function $\big(E \geq0 \mapsto q_E^\Omega(u)\big)$ is a continuous and concave, so is $\big(E\geq0 \mapsto \mu^\Omega(E)\big)$ and Point \eqref{itm:1} is proved.

Regarding Point \eqref{itm:2}, one observes that for all $u\in \dom(q_E^\Omega)$ there holds $q_0^\Omega(u) \geq 0$ and in particular $\mu^\Omega(0) \geq 0$. Now, for any $f\in L^2(\partial\Omega)$ we have $\Phi_{\rm h}(f) \in \dom(q_E^\Omega)$ and $q_0^\Omega(u) = 0$ because $\Phi_{\rm h}(f)$ is holomorphic in $\Omega$. Consequently, there holds $\mu^\Omega(0) = 0$.

To prove the second part of Point \eqref{itm:2}, let $u \in \dom(q_E^\Omega)$ and remark that
\begin{equation}
\label{eqn:minmaxpos}
	q_E^\Omega(u) = (4-E) \|\partial_{\bar z} u\|_{L^2(\Omega)}^2 - E^2 \|u\|_{L^2(\Omega)}^2 + E \mathfrak{Q}(u)
\end{equation}
where the quadratic form $\mathfrak{Q}$ is defined as
\[
	\mathfrak{Q}(u) = \|\partial_{\bar z} u\|_{L^2(\Omega)}^2 + \|u\|_{L^2(\partial\Omega)}^2,\quad \dom(\mathfrak{Q}) = \dom (q_E^\Omega).
\]
Now, remark that $\mathfrak{Q} \geq 0$ thus, by Kato's first representation theorem, there exists a unique self-adjoint operator $\mathfrak{H}$ such that $\dom(\mathfrak{H}) \subset \dom(\mathfrak{Q})$ and its spectrum is a sequence of non-decreasing eigenvalues because $\dom(\mathfrak{Q}) = \dom(q_E^\Omega)$ is compactly embedded into $L^2(\Omega)$. Let $\lambda_1^\Omega$ be its smallest eigenvalue, we already know by the min-max principle that $\lambda_1^\Omega \geq 0$. Moreover, if $\lambda_1^\Omega = 0$, for an associated eigenfunction $u$, we obtain $\mathfrak{Q}(u) = 0$ which implies that $\partial_{\bar z} u = 0$ hence $u$ is holomorphic with trace in $L^2(\partial\Omega)$. Consequently, $u$ belongs to $\mathcal{H}_{\rm h}^2(\Omega)$ and $u = \Phi_{\rm h}(f)$ for some $f \in L^2(\partial\Omega)$ such that $\Gamma^+ u = f$. However, as $\mathfrak{Q}(u) = 0$, we also obtain $\Gamma^+ u = f = 0$ which yields $u=0$ which is not possible because $u$ is an eigenfunction. It implies that $\lambda_1^\Omega > 0$ and using the min-max principle in \eqref{eqn:minmaxpos}, we get for all $u \in \dom(q_E^\Omega)$:
\[
	q_E^\Omega(u) \geq (4-E) \|\partial_{\bar z} u\|_{L^2(\Omega)}^2 - E^2 \|u\|_{\Omega}^2 + E \lambda_1^\Omega \|u\|_{L^2(\Omega)}^2.
\]
In particular, if $E< 4$ we obtain
\[
	q_E^\Omega(u) \geq E \big(\lambda_1^\Omega - E\big)\|u\|_{L^2(\Omega)}^2
\]
and the min-max principle yields
\[
	\mu^\Omega(E) \geq E(\lambda_1^\Omega - E).
\]
Thus, setting $E_\star^\Omega := \min(4,\lambda_1^\Omega)$, for all $E \in (0,E_\star^\Omega)$, we have $\mu^\Omega(E) >0$.

Let us prove Point \eqref{itm:4}. Let $u \in \dom(q_E^\Omega)$ and $0 < E_1 < E_2$. There holds
\begin{equation}\label{eqn:monotonicity}
q_{E_2}^\Omega (u) = q_{E_1}^\Omega(u) - (E_2^2 - E_1^2)\int_{\Omega}|u|^2 dx + (E_2 - E_1)\int_{\partial\Omega}|u|^2 ds.
\end{equation}
Now, pick $u_1$ a normalized eigenfunction of $H_{E_1}^\Omega$ associated with the eigenvalue $\mu^\Omega(E_1)$. We have $q_{E_1}^\Omega(u_1) = \mu^\Omega(E_1)$ which implies
\[
	\int_{\partial\Omega}|u_1|^2 ds \leq \frac{1}{E_1}\Big(4\int_\Omega |\partial_{\bar z}u_1|^2 dx + E_1 \int_{\partial\Omega}|u_1|^2ds\Big) = \frac1{E_1}(q_{E_1}^\Omega (u_1) + E_1^2 )\leq \frac{E_1^2 + \mu^\Omega(E_1)}{E_1}.
\]
Thus, evaluating \eqref{eqn:monotonicity} with $u = u_1$ we obtain
\[
	q_{E_2}^\Omega(u_1) \leq \mu^\Omega(E_1) - (E_2^2-E_1^2) + \frac{E_2 - E_1}{E_1}(E_1^2 + \mu^\Omega(E_1)).
\]
The min-max principle finally gives the sought inequality
\[
\begin{array}{lcl}
	\mu^\Omega(E_2) &\leq&\displaystyle \mu^\Omega(E_1) - (E_2^2-E_1^2) + \frac{E_2 - E_1}{E_1}(E_1^2 + \mu^\Omega(E_1))\\
				&=& \displaystyle\frac{E_2}{E_1}\mu^\Omega(E_1) - E_2(E_2 - E_1).
\end{array}
\]
Now, assume that $\mu^\Omega(E_1) = 0$. It yields
\[
	\mu^\Omega(E_2) \leq - E_2(E_2 - E_1) < 0.
\]
Similarly, if $\mu^\Omega(E_2) = 0$ we get
\[
	0 < E_1(E_2 - E_1) \leq \mu^\Omega(E_1).
\]
\end{proof}

\subsection{Proof of the variational principle}\label{subsec:para3}
In our way to prove Theorem \ref{thm:vf} we will need the following two propositions.
\begin{prop}Let $E>0$ be such that $\mu^\Omega(E) = 0$ then $E \in Sp_{dis}(D^\Omega)$.
\label{prop:sens1}
\end{prop}
\begin{proof} Let $E>0$ be such that $\mu^\Omega(E) = 0$ and consider a normalized associated eigenfunction $v \in \dom(H_E^\Omega)$. Set $u = (u_1,u_2)^\top = (v, -\frac{2 \rmi}{E}\partial_{\bar z}v)^\top$, by Proposition \ref{prop:domope}, $u\in H^1(\Omega,\C^2)$ and as $v\in\dom(H_E^\Omega)$, in $H^{\frac12}(\partial\Omega)$ there holds
\[
	\Gamma^+(\partial_{\bar z} v) +{\bf n}\frac{E}2 \Gamma^+ v = 0 \Longleftrightarrow -2E^{-1}\rmi \Gamma^+(\partial_{\bar z} v) = \rmi {\bf n} \Gamma^+u \Longleftrightarrow \Gamma^+ u_2 = \rmi {\bf n}\Gamma^+ u_1.
\]
Hence, $(u_1,u_2)^\top \in \dom(D^\Omega)$ and there holds
\begin{align*}
	D^\Omega (u_1,u_2)^\top = \begin{pmatrix} 0 & -2\rmi\partial_{z}\\-2\rmi\partial_{\bar z} & 0\end{pmatrix} (u_1,u_2)^\top &= (-2\rmi\partial_z u_2, -2\rmi\partial_{\bar z} u_1)^\top\\&= (-\frac1{E}\Delta u, E u_2)^\top\\
	& = E(u_1,u_2)^\top.
\end{align*}
Hence, $E \in Sp_{dis}(D_\Omega)$ and it concludes the proof of Proposition \ref{prop:sens1}.

\end{proof}
\begin{prop}Let $E \in Sp_{dis}(D^\Omega)\cap\mathbb{R}_+^*$ then $\mu^\Omega(E) \leq 0$.
\label{prop:sens2}
\end{prop}
\begin{proof}Let $E \in Sp_{dis}(D^\Omega)\cap\mathbb{R}_+^*$ and pick $u = (u_1,u_2)^\top \in \dom(D^\Omega)$ a normalized eigenfunction of $D^\Omega$ associated with $E$. We have
\[
	\left\{
		\begin{array}{ll}
			D^\Omega u = E u& \text{in } \Omega,\\
			u_2 = \rmi {\bf n} u_1 & \text{on } \partial\Omega.
		\end{array}
	\right.
\]
In particular, we have $-2\rmi \partial_{\bar z} u_1 = E u_2$ and $\partial_{\bar z}u_1 \in H^1(\Omega)$. It yields
\[
	E u_1 = -2\rmi \partial_{z} u_2 = - \frac4{E} \partial_{z}\partial_{\bar z} u_1.
\]
Taking the scalar product with respect to $u_1$ on both side of the previous equation we get
\begin{equation}\label{eqn:obtfq}
	E^2 \int_{\Omega}|u_1|^2 dx = - 4 \int_{\Omega} (\partial_{z}\partial_{\bar z} u_1)\overline{u_1}dx = 4\int_{\Omega}|\partial_{\bar z} u_1|^2dx - 2\int_{\partial\Omega}\overline{\bf n}(\partial_{\bar z}u_1) \overline{u_1} ds.
\end{equation}
Now, remark that on $\partial\Omega$, we have
\[
	-\frac{2\rmi}{E} \partial_{\bar z}u_1 = u_2 = \rmi {\bf n} u_1
\]
which implies that on $\partial\Omega$
\[
	2\overline{\bf n}\partial_{\bar z}u_1 + {E}u_1 = 0.
\]
Hence, \eqref{eqn:obtfq} becomes
\[
	E^2\int_{\Omega} |u_1|^2 = 4\int_{\Omega}|\partial_{\bar z}u_1|^2 dx + E \int_{\partial\Omega}|u_1|^2 ds
\]
which reads $q_{E}^\Omega(u_1) = 0$ thus, the min-max principle gives $\mu^\Omega(E)\leq 0$.
\end{proof}

Now, we have all the tools to prove Theorem \ref{thm:vf}. The proof is performed proving each implication.
\begin{proof}[Proof of Theorem \ref{thm:vf}]~
By Proposition \ref{prop:sens2}, we have $\mu^\Omega(E_1(\Omega)) \leq 0$. Assume that $\mu^\Omega(E_1(\Omega)) < 0$, by Proposition \ref{prop:monotonicity} we know that there exists $0<E < E_1(\Omega)$ such that $\mu^\Omega(E) = 0$ which, by Proposition \ref{prop:sens1}, implies $E \in Sp_{dis}(D^\Omega)$. It is not possible because, by definition of $E_1(\Omega)$, $E\geq E_1(\Omega)$ consequently, we obtain $\mu^\Omega(E_1(\Omega)) = 0$.

Let $E>0$ be such that $\mu^\Omega(E) = 0$. By Proposition \ref{prop:sens1}, $E \in Sp_{dis}(D^\Omega)$ and necessarily $E \geq E_1(\Omega)$. If $E>E_1(\Omega)$, by Proposition \ref{prop:monotonicity}, we obtain $\mu^\Omega(E_1(\Omega)) > 0$ but by Proposition \ref{prop:sens2} we necessarily have $\mu^\Omega(E_1(\Omega)) \leq 0$ which implies that necessarily there holds $E = E_1(\Omega)$.

\end{proof}
\section{Geometric upper bounds on the spectral gap}\label{sec:isopin}
The goal of this section is to prove Theorem \ref{thm:ineq} and this is discussed in \S \ref{subsec:sharpub}. But first, in \S \ref{subsec:simpleub}, we give a simple geometric upper bound on the spectral gap which illustrates how Theorem \ref{thm:vf} can be used.
\subsection{A simple upper bound}\label{subsec:simpleub}
An immediate consequence of Theorem \ref{thm:vf} reads as follows.
\begin{prop}
	Let $\Omega \subset \mathbb{R}^2$ be $C^\infty$ and simply connected. There holds
	\[
		E_1(\Omega) \leq \frac{|\partial\Omega|}{|\Omega|}.
	\]
\end{prop}
There is no reason for the above upper bound to be attained among Euclidean domains. However, the bound brings into play simple geometric quantities: the perimeter and the area of $\Omega$.
\begin{proof}
Let $E>0$ and $u \equiv 1$ the function constant to $1$ in $\Omega$. As $u\in \dom(q_E^\Omega)$, by the min-max principle we obtain
\[
	\mu^\Omega(E) \leq \frac{q_E^\Omega(u)}{\|u\|_{L^2(\Omega)}^2} = E\Big(\frac{|\partial\Omega|}{|\Omega|} - E\Big).
\]
So in $E_{\rm crit} := \frac{|\partial\Omega|}{|\Omega|}$ we get $\mu^\Omega(E_{\rm crit}) \leq 0$ and by Proposition \ref{prop:monotonicity} we know that
\[
	E_1(\Omega) \leq E_{\rm crit} = \frac{|\partial\Omega|}{|\Omega|}.
\]
\end{proof}

\subsection{A sharp upper bound}\label{subsec:sharpub}
It turns out Theorem \ref{thm:ineq} is a consequence of the following result.

\begin{thm} Let $\Omega \subset \R^2$ be a $C^\infty$ simply connected domain. There holds
\[
	E_1(\Omega) \leq \frac{|\partial\Omega| + \sqrt{|\partial\Omega|^2 + 8\pi E_1(\D)(E_1(\D) - 1)(\pi r_i^2 +|\Omega|)}}{2(\pi r_i^2 + |\Omega|)}
\]
with equality if and only if $\Omega$ is a disk.
\label{thm:isoper}
\end{thm}

Now, we have all the tools to prove Theorem \ref{thm:ineq}.
\begin{proof}[Proof of Theorem \ref{thm:ineq}] Using that $\pi r_i^2 \leq |\Omega|$ and the isoperimetric inequality we obtain $4\pi^2 r_i^2 \leq 4\pi |\Omega| \leq |\partial\Omega|^2$. It gives
\[
	|\partial\Omega|^2 + 8\pi E_1(\D)(E_1(\D) - 1)(\pi r_i^2 +|\Omega|) \leq  |\partial\Omega|^2(2E_1(\D) - 1)^2.
\]
Note that in the above inequalities, we have equality if and only if $\Omega$ is a disk and combining this bound with the one of Theorem \ref{thm:isoper} we get Theorem \ref{thm:ineq}. 
\end{proof}

In the rest of this section we focus on proving Theorem \ref{thm:isoper} and assume, without loss of generality, the following.
\begin{enumerate}[label=(\roman*)]
	\item $0\in \Omega$ is such that $r_i = \max_{x\in\partial\Omega}|x|$,
	\item $f : \D \to \Omega$ is a conformal map such that $f(0) = 0$ and we write
\[
	f(z) =  \sum_{n\geq1} c_n z^n,
\]
where $(c_n)_{n\geq1}$ is a sequence of complex numbers.
\end{enumerate}

Before going through the proof of Theorem \ref{thm:isoper}, we gather in the following paragraph some known properties linking the geometry of $\Omega$ with the conformal map $f$.

\subsubsection{Preliminaries}

The next proposition can be found in \cite[\S 3.10.2]{P16} and relates the area of $\Omega$ with the conformal map $f$.

\begin{prop}[Area formula]\label{eqn:area} There holds
\begin{equation*}
	|\Omega| = \pi \sum_{n\geq 1}n|c_n|^2.
\end{equation*}
\end{prop}

The second proposition is a consequence of the Schwarz lemma (see Koebe's estimate in \cite[Chap. I, Thm. 4.3]{GM05}). It gives a relation between the first coefficient $c_1$ of the conformal map $f$ and the inradius $r_i$.

\begin{prop}[Koebe's estimate]\label{prop:Koebe} There holds
\[
	|f'(0)| = |c_1| \geq r_i.
\]
\end{prop}

Finally, the last geometric relation between the conformal map $f$ and the geometry of $\Omega$ we need to prove Theorem \ref{thm:isoper} is that the perimeter $|\partial\Omega|$ of $\Omega$ can be expressed as
\begin{equation}
	|\partial\Omega| = \int_{0}^{2\pi}|f'(e^{i\theta})|d\theta.
\label{eqn:perim}
\end{equation}
\eqref{eqn:perim} is a simple consequence of the fact that $f|_{\mathbb{S}^1}$ is a parametrization of $\partial\Omega$.
\subsubsection{Proof of the upper bound on the spectral gap}

To prove Theorem \ref{thm:isoper}, we construct an adequate test function for $q_E^\Omega$ transplanting the eigenfunction of the unit disk $\D$ in the domain $\Omega$ thanks to the conformal map $f$. We obtain an upper bound on $\mu^\Omega(E)$ which is a second order polynomial in the spectral parameter $E>0$ and with coefficients depending on the geometry of $\Omega$. It translates into an optimization problem for the spectral parameter $E>0$ that we solve in the last step of the proof.

\begin{proof}[Proof of Theorem \ref{thm:isoper}]

Let us go through all the steps of the proof.\\

\paragraph{Step 1} Let us denote by $J_0$ (resp. $J_1$) the Bessel function of the first kind of order $0$ (resp. of order $1$). For $x \in \D$, consider $u_0(x) = J_0\big(E_1(\D) |x|\big) \in H^1(\D) \subset \dom(q_{E_1(\D)}^\Omega)$. As explained in Remark \ref{rem:fundisk} $u(x) = (u_0(x), \rmi \frac{x_1 + ix_2}{|x|}J_1\big(E_1(\D) |x|\big))^\top$ is an eigenfunction of $D^\D $associated with $E_1(\D)$. Theorem \ref{thm:vf} implies
\begin{align}\nonumber
0 = q_{E_1(\D)}^\D(u_0) = &\  2\pi E_1(\D)^2\int_0^1 J_1\big(E_1(\D) r\big)^2 rdr - 2\pi E_1(\D)^2\int_0^1 J_0\big(E_1(\D) r\big)^2 r dr\\&\quad + 2\pi E_1(\D) J_0\big(E_1(\D)\big)^2.
\label{eqn:carvpdisk}
\end{align}

\paragraph{Step 2} For $x=(x_1,x_2)\in \Omega$, consider $v_0(x_1,x_2) = u_0(f^{-1}(x_1 +\rmi x_2)) \in H^1(\Omega) \subset \dom(q_E^\Omega)$. By the min-max principle, there holds
\begin{equation}
	\mu^\Omega(E) \leq \frac{q_E^\Omega(v_0)}{\|v_0\|_{L^2(\Omega)}^2} = \frac{\|\nabla v_0\|_{L^2(\Omega)}^2 + E \|v_0\|_{L^2(\partial\Omega)}^2}{\|v_0\|_{L^2(\Omega)}^2} - E^2,
	\label{eqn:ub1step}
\end{equation}
where we have used that $v_0$ is real valued to ensure that $\|\nabla v_0\|_{L^2(\Omega)} = 4\|\partial_{\bar z}v_0\|_{L^2(\Omega)}$.
\paragraph{Step 3} Now, as $f$ is a conformal map, we know that
\begin{equation}
	\|\nabla v_0\|_{L^2(\Omega)} = \|\nabla u_0\|_{L^2(\D)}  = 2 \pi E_1(\D)^2 \int_0^1 J_1\big(E_1(\D) r\big)^2 r dr.
\label{eqn:eqnfinal1}
\end{equation}
Using \eqref{eqn:perim}, we obtain
\begin{equation}
	\|v_0\|_{L^2(\partial\Omega)}^2 = \int_{0}^{2\pi} |v_0(f(e^{\rmi \theta}))|^2 |f'(e^{\rmi \theta})| d\theta = J_0\big(E_1(\D)\big)^2 |\partial \Omega|.
	\label{eqn:eqnfinal2}
\end{equation}
Finally, the last integral reads
\begin{align}\nonumber
	\|v_0\|_{L^2(\Omega)}^2 =& \int_0^1\int_0^{2\pi}|u_0(r)|^2 |f'(re^{\rmi \theta})|^2 r dr d\theta\\\nonumber=& \int_0^1 |u_0(r)|^2 \Big(\int_0^{2\pi}\Big|\sum_{n\geq 1} n c_n r^{n-1} e^{\rmi (n-1)\theta}\Big|^2 d\theta\Big) r dr\\
	=& 2 \pi \sum_{n\geq 1} n |c_n|^2 M_n,\quad \text{where for } n\geq 1, M_n := n\int_{0}^1J_0\big(E_1(\D)r\big)^2 r^{2n-1} dr,
	\label{eqn:eqnfinal3}
\end{align}
where we have used Parseval identity.

\paragraph{Step 4}
Taking into account \eqref{eqn:eqnfinal1},\eqref{eqn:eqnfinal2} and \eqref{eqn:eqnfinal3}, \eqref{eqn:ub1step} becomes
\begin{align}
	\nonumber\mu^\Omega(E) \leq  2\pi E_1(\D)^2 \frac{\displaystyle\int_{0}^1J_1\big(E_1(\D) r\big)^2 rdr}{\displaystyle2\pi \sum_{n\geq 1}n|c_n|^2M_n} &\ - E^2\\&\ + E\frac{J_0\big(E_1(\D)\big)^2|\partial\Omega|}{\displaystyle2\pi \sum_{n\geq 1}n|c_n|^2M_n}.\label{eqn:ub1}
	\end{align}

Let us find a lower bound on the sequence $(M_n)_{n\geq 1}$. Using first an integration by parts we find
\[
	M_n = \frac12J_0\big(E_1(\D)\big)^2 + \frac{E_1(\D)}2\int_{0}^1J_0\big(E_1(\D) r\big)J_1\big(E_1(\D) r\big) r^{2n} dr.
\]
In particular, for $n=1$ it gives
\begin{align}
	M_1 = \int_0^1 J_0\big(E_1(\D) r\big)^2 rdr & =J_0\big(E_1(\D) \big)^2\label{eqn:intJ0}  \\\nonumber & = E_1(\D) \int_{0}^1 J_0\big(E_1(\D) r\big) J_1(E_1(\D) r) r^2 dr.
\end{align}
Now, for $n\geq1$, one notices that $h_1 := \Big(r \mapsto (J_0J_1)\big(E_1(\D)r\big)r^2\Big)$ and $h_2 := \Big(r \mapsto r^{2n-2}\Big)$ are non-decreasing functions on $[0,1]$ and by Chebyschev's inequality for non-decreasing functions, we obtain
\[
	M_n \geq \frac12 M_1 + \frac12M_1 \int_{0}^1r^{2n-2}dr = \frac{n}{2n-1}M_1.
\]
In particular, we have
\begin{align}
	\nonumber 2\pi \sum_{n\geq 1}n |c_n|^2 M_n &\geq J_0\big(E_1(\D)\big)^2\Big(2\pi |c_1|^2 + 2\pi\sum_{n\geq 2}\frac{n^2}{2n-1}|c_n|^2\Big)\\\nonumber& \geq J_0\big(E_1(\D)\big)^2\Big(2\pi |c_1|^2 + \pi\sum_{n\geq 2}n|c_n|^2\Big) \\\nonumber
	& = J_0\big(E_1(\D)\big)^2 (\pi |c_1|^2 + |\Omega|)\\
	& \label{eqn:lb1}\geq J_0\big(E_1(\D)\big)^2(\pi |r_i|^2 + |\Omega|),
\end{align} 
where we have used Proposition \ref{eqn:area} and Proposition \ref{prop:Koebe}. Remark that in the first two inequalities above we have equality if and only if $c_n = 0$ for all $n\geq2$. Similarly, in the last equality, we have equality if and only if $|c_1| = r_i$. In particuliar there is equality in the above inequalities if and only if $f(z) = c_1 z$ and $\Omega$ is a disk centered in $0$ of radius $r_i$.

Combining \eqref{eqn:carvpdisk} and \eqref{eqn:lb1} in \eqref{eqn:ub1}, we obtain

\begin{multline*}
	\mu^\Omega(E) \leq - E^2 + \frac{2 \pi E_1(\D)^2 \int_0^1 J_0\big(E_1(\D)r\big)^2 rdr + J_0\big(E_1(\D)\big)^2\Big( E |\partial\Omega| - 2\pi E_1(\D)\Big)}{J_0\big(E_1(\D)\big)^2(\pi r_i^2 + |\Omega|)}.
\end{multline*}
Using \eqref{eqn:intJ0}, we obtain
\begin{align*}
	\mu^\Omega(E) &\leq -E^2 + \frac{2\pi E_1(\D)^2 + \big(E |\partial\Omega| - 2\pi E_1(\D) \big)}{\pi r_i^2 + |\Omega|}\\  &= \frac{\big(2\pi E_1(\D)^2 - (\pi r_i^2 + |\Omega|)E^2\big) + \big(E |\partial\Omega| - 2\pi E_1(\D)\big)}{\pi r_i^2 + |\Omega|}\\ &= \frac{P(E)}{\pi r_i^2 + |\Omega|}, \quad P(E) := -E^2(\pi r_i^2 + |\Omega|)+ E |\partial\Omega| +2\pi E_1(\D)\big(E_1(\D)-1\big).
\end{align*}
\paragraph{Step 5}
Remark that by \eqref{eqn:lbBFSVdB}, there holds $E_1(\D) -1 \geq \sqrt{2} - 1 >0$. In particular, the discriminant of $P$ satisfies
\[
	\delta(P) := |\partial\Omega|^2 + 8\pi E_1(\D)\big(E_1(\D) - 1\big)(\pi r_i^2 +|\Omega|)>0.
\]
Thus, $P$ has two real roots and as $P(0)>0$, the only positive root is
\[
	E_{\rm crit} := \frac{|\partial\Omega| + \sqrt{|\partial\Omega|^2 + 8\pi E_1(\D)\big(E_1(\D) - 1\big)(\pi r_i^2 + |\Omega|)}}{2(\pi r_i^2 +|\Omega|)}.
\]
One obtains $\mu^\Omega(E_{\rm crit}) \leq \frac{P(E_{\rm crit})}{\pi r_i^2 + |\Omega|} = 0$ and by Proposition \ref{prop:monotonicity} and Theorem \ref{thm:vf} we get
\[
	E_1(\D) \leq E_{\rm crit}
\]
which is precisely Theorem \ref{thm:isoper}.
\end{proof}

\section{About the Faber-Krahn conjecture}\label{sec:aboutFK}
In this section we discuss how the variational formulation established in Theorem \ref{thm:vf} can be used to investigate Conjecture \ref{conj:FK}. \S \ref{par:newconj} deals with a new Faber-Krahn type conjecture for the operator $H_E^\Omega$ introduced in \S \ref{subsec:para1} and how this new conjecture is related to Conjecture \ref{conj:FK}. In \S \ref{par:BosselDaners}, we discuss how the well-known Bossel-Daners inequality for the Robin Laplacian is linked to Conjecture \ref{conj:FK} (see \cite{Boss86,Dan06}).
\subsection{A new conjecture}\label{par:newconj}
Let us introduce a new Faber-Krahn type conjecture for $\mu^\Omega(E)$, the first eigenvalue of $H_E^\Omega$.
\begin{conj}Let $\Omega \subset \mathbb{R}^2$ be  $C^\infty$ and simply connected. For all $E>0$, there holds
\[
	\mu^\Omega(E) \geq \frac{\pi}{|\Omega|}\mu^\D\Big(\sqrt{\frac{|\Omega|}{\pi}}E\Big).
\]
Moreover, there is equality in the above inequality if and only if $\Omega$ is a disk.
\label{conj:2}
\end{conj}
It turns out Conjecture \ref{conj:2} is equivalent to Conjecture \ref{conj:FK} and this is what we prove in the rest of this paragraph.
\begin{proof} First, remark that a simple scaling argument gives, for all $E>0$, that
\[
	\sqrt{\frac{\pi}{|\Omega|}}E_1(\D) = E_1(\rho \D), \quad \mu^{\rho \D}(E) = \frac{\pi}{|\Omega|}\mu^{\D}\Big(\sqrt{\frac{|\Omega|}{\pi}}E\Big)\quad \text{where } \rho := \sqrt{\frac{|\Omega|}{\pi}}.
\]
Second, assume that Conjecture \ref{conj:FK} holds true. If $\Omega$ is a disk, there holds $\mu^\Omega(E) = \mu^{\rho\D}(E)$ so now, we assume that $\Omega$ is not a disk. Let us prove that for all $E>0$ there holds
\[
	\mu^\Omega(E) > \mu^{\rho \D}(E).
\]
Let us reason by \textit{reduction ad absurdum} and assume there exists $E_\star > 0$ such that $\mu^\Omega(E_\star) \leq \mu^{\rho \D}(E_\star)$.

\underline{Case $E_\star < E_1(\rho\D)$.} By hypothesis and Proposition \ref{prop:monotonicity}, there holds
\begin{align*}
	\mu^\Omega(E_\star) \leq \mu^{\rho \D}(E_\star) &\leq \frac{E_1(\rho\D)}{E_\star} \mu^{\rho\D}\big(E_1(\rho\D)\big) - E_1(\rho\D)(E_1(\rho\D) - E_\star)\\
	& = - E_1(\rho\D)(E_1(\rho\D) - E_\star) < 0.
\end{align*}
In particular, $\mu^\Omega(E_\star) < 0$ which implies $E_\star > E_1(\Omega)$. However, if Conjecture \ref{conj:FK} holds true we obtain $E_\star > E_1(\Omega) > E_1(\rho \D)$ which contradicts our hypothesis.

\underline{Case $E_1(\rho\D) \leq E_\star \leq E_1(\Omega)$.} By hypothesis and Proposition \ref{prop:monotonicity}, there holds
\[
	0 \leq \mu^\Omega(E_\star) \leq \mu^{\rho \D}(E_\star) \leq 0,
\]
which contradicts our hypothesis because we obtain $E_\star = E_1(\Omega) = E_1(\rho\D)$ but we have assumed that $\Omega$ is not a disk thus, this equality can not hold if Conjecture \ref{conj:FK} holds true.

\underline{Case $E_\star > E_1(\Omega)$.} By hypothesis and Proposition \ref{prop:monotonicity}, there holds
\begin{align*}
	0 = \mu^\Omega\big(E_1(\Omega)\big) &\leq \frac{E_\star}{E_1(\Omega)}\mu^\Omega(E_\star) - E_\star\big(E_\star - E_1(\Omega)\big)\\& \leq \frac{E_\star}{E_1(\Omega)}\mu^{\rho \D}(E_\star) - E_\star\big(E_\star - E_1(\Omega)\big).
\end{align*}
In particular, we obtain $\mu^{\rho \D}(E_\star) \geq E_1(\Omega)\big(E_\star -E_1(\Omega)\big) > 0$. Hence, $E_\star < E_1(\rho \D)$ which contradicts Conjecture \ref{conj:FK}.

Consequently, we have proved that if Conjecture \ref{conj:FK} holds true so does Conjecture \ref{conj:2}.

Finally, let us assume that Conjecture \ref{conj:2} holds true. If $\Omega$ is a disk, we obtain that for all $E> 0$, $\mu^\Omega(E) = \mu^{\rho\D}(E)$. In particular, in $E = E_1(\Omega)$ we get $\mu^{\rho\D}\big(E_1(\Omega)\big) = 0$ and $E_1(\rho \D) = E_1(\Omega)$.

When $\Omega$ is not a disk, for all $E>0$ there holds $\mu^{\rho \D}(E) < \mu^\Omega(E)$. In $E=E_1(\Omega) $ we obtain $\mu^{\rho\D}\big(E_1(\Omega)\big) < 0$ and by Proposition \ref{prop:monotonicity} we obtain $E_1(\rho\D) < E_1(\Omega)$ which is precisely Conjecture \ref{conj:FK}.
\end{proof}
\subsection{Link with the Bossel-Daners inequality}\label{par:BosselDaners}

The first eigenvalue of the Robin Laplacian with positive parameter $E>0$ in the domain $\Omega$, denoted $\lambda_{\rm Rob}^\Omega(E)$, is given by the variational characterization
\[
	\lambda_{\rm Rob}^\Omega(E) := \inf_{u \in C^\infty(\overline{\Omega})\setminus\{0\}} \frac{\|\nabla u\|_{L^2(\Omega)}^2 + E \int_{\partial\Omega}|u|^2 ds}{\|u\|_{L^2(\Omega)}^2}
\]
and the Bossel-Daners inequality states that
\begin{equation}
	\lambda_{\rm Rob}^\Omega(E) \geq \frac{\pi}{|\Omega|}\lambda_{\rm Rob}^\D\Big(\sqrt{\frac{|\Omega|}{\pi}}E\Big),
	\label{eqn:Boss-Dan}
\end{equation}
with equality if and only if $\Omega$ is a disk.
Note that the structure of \eqref{eqn:Boss-Dan} is similar to that of Conjecture \ref{conj:2} and it turns out they are intimately connected. This is the purpose of the following proposition.

\begin{prop} Conjecture \ref{conj:FK} implies the Bossel-Daners inequality \eqref{eqn:Boss-Dan}.
\end{prop}
\begin{proof} As Conjecture \ref{conj:FK} is equivalent to Conjecture \ref{conj:2} as discussed in \S \ref{par:newconj}, we can assume that Conjecture \ref{conj:2} holds. Let us start by remarking that for all $E>0$, if $u\in \dom(H_E^\D)$ is a normalized eigenfunction associated with $\mu^{\D}(E)$ then $u$ can be picked real-valued. Hence, we get
\begin{align}
	\nonumber\mu^{\D}(E) &= \inf_{v \in C^\infty(\overline{\D},\mathbb{R})} \frac{\|\nabla v\|_{L^2(\D)}^2 - E^2\|v\|_{L^2(\D)}^2 + \int_{\partial\D}|v|^2 ds}{\|v\|_{L^2(\D)}^2}\\
	& = \lambda_{\rm Rob}^{\D}(E) - E^2.\label{eqn:minfin2}
\end{align}
Now, we remark that for any domain $\Omega$ there holds
\begin{align}\nonumber
	\lambda_{\rm Rob}^\Omega(E) - E^2 &= \inf_{v \in C^\infty(\overline{\Omega},\mathbb{R})\setminus\{0\}} \frac{\|\nabla v\|_{L^2(\Omega)}^2 - E^2\|v\|_{L^2(\Omega)}^2 + E \int_{\partial\Omega}|v|^2 ds}{\|v\|_{L^2(\Omega)}^2}\\&\nonumber = \inf_{v \in C^\infty(\overline{\Omega},\mathbb{R})\setminus\{0\}} \frac{4\|\partial_{\bar z} v\|_{L^2(\Omega)}^2 - E^2\|v\|_{L^2(\Omega)}^2 + E \int_{\partial\Omega}|v|^2 ds}{\|v\|_{L^2(\Omega)}^2}\\
	& \nonumber\geq \inf_{v \in \dom(q_E^\Omega))\setminus\{0\}} \frac{4\|\partial_{\bar z} v\|_{L^2(\Omega)}^2 - E^2\|v\|_{L^2(\Omega)}^2 + E \int_{\partial\Omega}|v|^2 ds}{\|v\|_{L^2(\Omega)}^2}\\
	& = \mu^\Omega(E).\label{eqn:eqmin1fin}
\end{align}
Hence, using \eqref{eqn:minfin2} and \eqref{eqn:eqmin1fin}, we get
\[
	\lambda_{\rm Rob}^\Omega(E) - E^2 \geq \mu^\Omega(E) \geq \frac{\pi}{|\Omega|}\mu^\D\Big(\sqrt{\frac{|\Omega|}{\pi}}E\Big) = \frac{\pi}{|\Omega|}\lambda_{\rm Rob}^\D\Big(\sqrt{\frac{|\Omega|}{\pi}}E\Big) - E^2.
\]
If $\Omega$ is a disk, all the above inequalities are equalities. Else, we obtain
\[
	\lambda_{\rm Rob}^\Omega(E) > \frac{\pi}{|\Omega|}\lambda_{\rm Rob}^\D\Big(\sqrt{\frac{|\Omega|}{\pi}}E\Big),
\]
which is precisely the Bossel-Daners inequality \eqref{eqn:Boss-Dan}.
\end{proof}

\section{Numerics}\label{sec:numerics}
The goal of this section is to illustrate numerically some theoretical results discussed in the previous sections and to support the validity of Conjecture \eqref{conj:FK}.

In \S \ref{nummet}, we discuss the two numerical schemes we have employed in \S \ref{numres} in order to study the principal eigenvalue of the Dirac operator with infinite mass boundary conditions in various domains $\Omega$. We also discuss the structure of the associated eigenfunctions.
\subsection{Numerical Methods}
\label{nummet}
In this paragraph we present a brief description of the numerical methods that we use in this work.

We have implemented two different numerical approaches, respectively to calculate the eigenvalues of the Dirac operator with infinite mass boundary conditions, directly from the formulation of the eigenvalue problem and to solve the minimization problem associated with the non-linear variational characterization \eqref{eqn:firstminmax}, defining $\mu^\Omega(E)$.

The eigenvalues of the Dirac operator with infinite mass boundary conditions are calculated using a numerical method based on Radial Basis Functions (RBF) (see \textit{eg}.~\cite{Kansa,Fornberg}). We have chosen a set of $RBF$ centers $y_1,...,y_N\in\mathbb{R}^2$, for some $N\in\mathbb{N}$, which are generated by a node repel algorithm (see~\cite{A19} for details). The eigenfunction $u=(u_1,u_2)^\top$ is defined in $H^1(\Omega,\mathbb{C}^2)$ and we use the notation $u_1=v_1+iw_1$ and $u_2=v_2+iw_2$, where $v_1$, $w_1$ and $v_2$, $w_2$ are the real and imaginary parts of $u_1$ and $u_2$, respectively. The RBF numerical approximation for each of these functions is defined by
\begin{equation}
\label{rbf}
\begin{array}{c}
v_1(x)=\sum_{j=1}^N\alpha^{(1)}_j\phi_j(x),\quad w_1(x)=\sum_{j=1}^N\beta^{(1)}_j\phi_j(x),\\ v_2(x)=\sum_{j=1}^N\alpha^{(2)}_j\phi_j(x),\quad w_2(x)=\sum_{j=1}^N\beta^{(2)}_j\phi_j(x),\quad
\end{array}
\end{equation}
where $\phi_j(x)=\phi(|x-y_j|)$, for some function $\phi:\mathbb{R}_0^+\rightarrow\mathbb{R}$. Several $RBF$ functions can be considered (eg.~\cite{Fornberg,A19}), but in this work we consider the multiquadric one $\phi(r)=\sqrt{1+(\epsilon r)^2}$, for some $\epsilon>0$.

The eigenvalue problem for the Dirac operator with infinite mass boundary conditions can be written as

\[ \left\{ \begin{array}{cl}
 -\frac{\partial v_2}{\partial x_2}+\frac{\partial w_2}{\partial x_1}+\rmi\left(-\frac{\partial v_2}{\partial x_1}-\frac{\partial w_2}{\partial x_2}\right) = E \left(v_1+\rmi w_1\right) & \text{ in } \Omega \\
  \frac{\partial w_1}{\partial x_1}+\frac{\partial v_1}{\partial x_2}+\rmi\left(-\frac{\partial v_1}{\partial x_1}+\frac{\partial w_1}{\partial x_2}\right) = E \left(v_2+\rmi w_2\right) & \text{ in } \Omega \\
  \left(v_2+\rmi w_2\right) = \rmi (n_1+\rmi n_2)(v_1+\rmi w_1) & \text{ on } \partial \Omega 
\end{array} \right.\]
and splitting in real and imaginary parts we have
\begin{equation}
    \label{equationsnm}
 \left\{ \begin{array}{cl}
 -\frac{\partial v_2}{\partial x_2}+\frac{\partial w_2}{\partial x_1} = E v_1 & \text{ in } \Omega \\
-\frac{\partial v_2}{\partial x_1}-\frac{\partial w_2}{\partial x_2} = E w_1 & \text{ in } \Omega \\
  
  \frac{\partial w_1}{\partial x_1}+\frac{\partial v_1}{\partial x_2} = E v_2 & \text{ in } \Omega \\
 -\frac{\partial v_1}{\partial x_1}+\frac{\partial w_1}{\partial x_2} = E w_2 & \text{ in } \Omega \\
 
  v_2 = -n_1w_1-n_2v_1 & \text{ on } \partial \Omega \\
w_2=n_1v_1-n_2w_1 & \text{ on } \partial \Omega 
\end{array} \right.
\end{equation}

These equations are imposed at a discrete set of interior and boundary points. We consider $M^{\partial\Omega}\in\mathbb{N}$ points $p_1,...,p_{M^{\partial\Omega}}$ uniformly distributed on $\partial\Omega$ and $M^{\Omega}\in\mathbb{N}$ points $q_1,...,q_{M^\Omega}$ located at a grid defined on $\Omega$. Then, we calculate the matrices
\[\mathbf{M}^\Omega=\begin{bmatrix}
\phi_1(q_1) & \cdots & \phi_N(q_1)\\
\vdots & \ddots & \vdots\\
\phi_1(q_{M^{\Omega}}) & \cdots & \phi_N(q_{M^{\Omega}})
\end{bmatrix},\quad \mathbf{M}_1^\Omega=\begin{bmatrix}
\partial_1\phi_1(q_1) & \cdots & \partial_1\phi_N(q_1)\\
\vdots & \ddots & \vdots\\
\partial_1\phi_1(q_{M^{\Omega}}) & \cdots & \partial_1\phi_N(q_{M^{\Omega}})
\end{bmatrix},\]
\[\mathbf{M}_2^\Omega=\begin{bmatrix}
\partial_2\phi_1(q_1) & \cdots & \partial_2\phi_N(q_1)\\
\vdots & \ddots & \vdots\\
\partial_2\phi_1(q_{M^{\Omega}}) & \cdots & \partial_2\phi_N(q_{M^{\Omega}})
\end{bmatrix},\quad\mathbf{M}^{\partial\Omega}=\begin{bmatrix}
\phi_1(p_1) & \cdots & \phi_N(p_1)\\
\vdots & \ddots & \vdots\\
\phi_1(p_{M^{\partial\Omega}}) & \cdots & \phi_N(p_{M^{\partial\Omega}})
\end{bmatrix}.\]
and
\[\mathbf{M}_1^{\partial\Omega}=\begin{bmatrix}
n_1(p_1)\phi_1(p_1) & \cdots & n_1(p_1)\phi_N(p_1)\\
\vdots & \ddots & \vdots\\
n_1(p_{M^{\partial\Omega}})\phi_1(p_{M^{\partial\Omega}}) & \cdots & n_1(p_{M^{\partial\Omega}})\phi_N(p_{M^{\partial\Omega}})
\end{bmatrix},\]
\[\mathbf{M}_2^{\partial\Omega}=\begin{bmatrix}
n_2(p_1)\phi_1(p_1) & \cdots & n_2(p_1)\phi_N(p_1)\\
\vdots & \ddots & \vdots\\
n_2(p_{M^{\partial\Omega}})\phi_1(p_{M^{\partial\Omega}}) & \cdots & n_2(p_{M^{\partial\Omega}})\phi_N(p_{M^{\partial\Omega}})
\end{bmatrix}\]

Taking into account the definitions of the RBF linear combinations \eqref{rbf}, the numerical approximations for the eigenvalues are the values $E$ for which we have nonzero solutions of the overdetermined system of linear equations
\begin{equation}
    \label{matric}\begin{bmatrix}\mathbf{0}\\
\mathbf{0}\\
\mathbf{0}\\
\mathbf{0}\\
\mathbf{0}\\
\mathbf{0}
\end{bmatrix}=
\left(\begin{bmatrix}
\mathbf{0} & \mathbf{0} & -\mathbf{M}_2^\Omega &\mathbf{M}_1^\Omega\\
\mathbf{0} & \mathbf{0} & -\mathbf{M}_1^\Omega &-\mathbf{M}_2^\Omega\\
\mathbf{M}_2^\Omega &\mathbf{M}_1^\Omega&\mathbf{0} & \mathbf{0} \\
-\mathbf{M}_1^\Omega &\mathbf{M}_2^\Omega&\mathbf{0} & \mathbf{0} \\
\mathbf{M}_2^{\partial\Omega}&\mathbf{M}_1^{\partial\Omega}&\mathbf{M}^{\partial\Omega}&\mathbf{0}\\
-\mathbf{M}_1^{\partial\Omega}&\mathbf{M}_2^{\partial\Omega}&\mathbf{0}&\mathbf{M}^{\partial\Omega}
\end{bmatrix}-E\begin{bmatrix}
\mathbf{M}^\Omega&\mathbf{0}&\mathbf{0}&\mathbf{0}\\
\mathbf{0}&\mathbf{M}^\Omega&\mathbf{0}&\mathbf{0}\\
\mathbf{0}&\mathbf{0}&\mathbf{M}^\Omega&\mathbf{0}\\
\mathbf{0}&\mathbf{0}&\mathbf{0}&\mathbf{M}^\Omega\\
\mathbf{0}&\mathbf{0}&\mathbf{0}&\mathbf{0}\\
\mathbf{0}&\mathbf{0}&\mathbf{0}&\mathbf{0}
\end{bmatrix}\right).\begin{bmatrix}\mathbf{\alpha}^{(1)}\\
\mathbf{\beta}^{(1)}\\
\mathbf{\alpha}^{(2)}\\
\mathbf{\beta}^{(2)}
\end{bmatrix}.
\end{equation}

The numerical solution of the minimization problem associated to the non-linear variational characterization is obtained directly from \eqref{eqn:firstminmax}, defining the function
\[\mathcal{F}(\alpha_1^{(1)},...,\alpha_N^{(1)},\beta_1^{(1)},...,\beta_N^{(1)})=\frac{4 \int_\Omega |\partial_{\bar z} u_1|^2 dx - E^2 \int_{\Omega}|u_1|^2dx + E \int_{\partial\Omega} |u_1|^2 ds}{\int_\Omega |u_1|^2 dx}\]
that we minimize by a gradient type method. We refer to~\cite{A19} for details about the numerical quadratures to approximate the boundary and volume integrals in the definition of $\mathcal{F}$.

\subsection{Numerical Results}
\label{numres}
We start by testing our numerical algorithm for the calculation of the eigenvalues of the Dirac operator with infinite mass boundary conditions in the case of the unit disk, for which we know that the principal eigenvalue $E_1(\mathbb{D})$ is the smallest non-negative solution of the equation
\[J_0(\mu)=J_1(\mu)\]
and we have $E_1(\mathbb{D})=1.434695650819...$ In Table~\ref{table:numerrors} we show the absolute errors of the numerical approximations for the principal eigenvalue $E_1(\mathbb{D})$, for several choices of $\epsilon$ and $N$ and show that the numerical method can be highly accurate, even with a moderate value of $N$.

\begin{table}[!ht]
\centering 
\begin{tabular}{|c|c|c|c|}
\hline
& N=242 & N=323 & N=402 \\
\hline$\epsilon=5$& $4.45\times10^{-7}$ & $8.55\times10^{-8}$& $1.33\times10^{-8}$\\
\hline$\epsilon=10$& $1.30\times10^{-5}$ &$2.78\times10^{-6}$ &$4.93\times10^{-8}$ \\
\hline$\epsilon=15$& $4.92\times10^{-5}$ & $9.21\times10^{-6}$& $1.16\times10^{-6}$\\
\hline
\end{tabular}
\caption{Absolute errors of the numerical approximations for the principal eigenvalue $\lambda_1(\mathbb{D})$, for several choices of $\epsilon$ and $N$.}
\label{table:numerrors}
\end{table}
We have computed the principal eigenvalue for 2500 domains (with smooth boundary) randomly generated satisfying $|\Omega|=\pi$. The corresponding eigenvalues are plotted in Figure~\ref{fig:figure2}, as a function of the perimeter. We observe that the  principal eigenvalue is minimized for the domain which also minimizes the perimeter. By the classical isoperimetric inequality it is well know that for fixed area, the perimeter is minimized by the ball. Thus, these numerical results suggest that the Faber-Krahn type inequality stated in Conjecture \ref{conj:FK} shall hold for the Dirac operator with infinite mass boundary conditions.
\begin{figure}[!ht]
\includegraphics[scale=0.6]{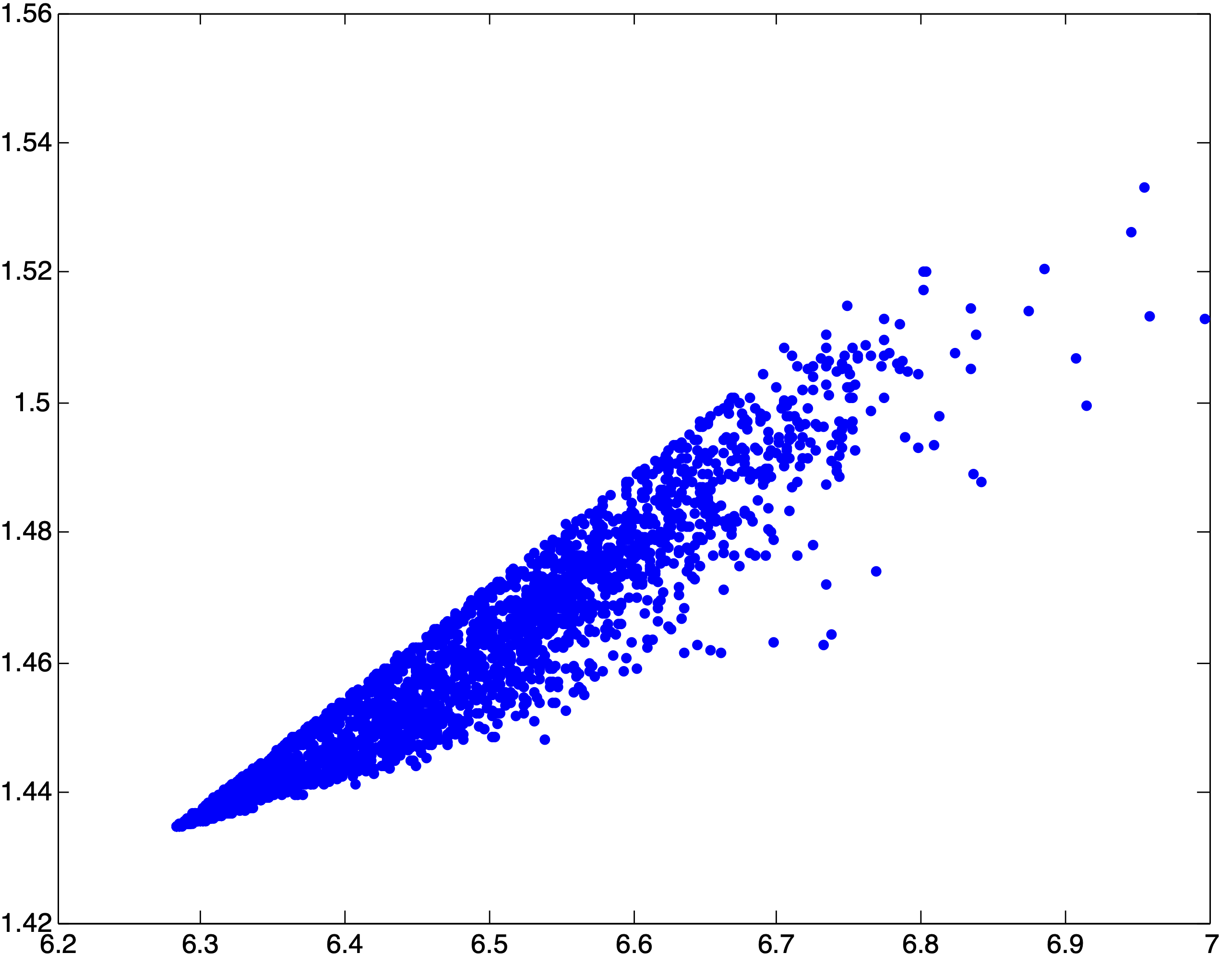}
\caption{Plot of the principal eigenvalue for 2500 domains (with smooth boundary) randomly generated satisfying $|\Omega|=\pi$, as a function of the perimeter.}
\label{fig:figure2}
\end{figure}

Next, we present some numerical results for the minimization problem associated to the non-linear variational characterization \eqref{eqn:firstminmax}. Figure~\ref{fig:figure3} shows three domains (denoted by $\Omega_1$, $\Omega_2$ and $\Omega_3$) verifying $|\Omega_i|=\pi,\ (i=1,2,3)$ to illustrate the numerical results that we gathered. In Figure~\ref{fig:figure4} we plot $\mu^{\Omega_i}(E),\ i=1,2,3$ together with the curve $\mu^{\mathbb{D}}(E)$. We verify that for all $E>0$, we have
\[
	\mu^{\Omega_i}(E) \geq \mu^\D(E),\ i=1,2,3
\] 
which illustrates Conjecture~\ref{conj:2}.

\begin{figure}[!ht]
\includegraphics[width=0.32\textwidth]{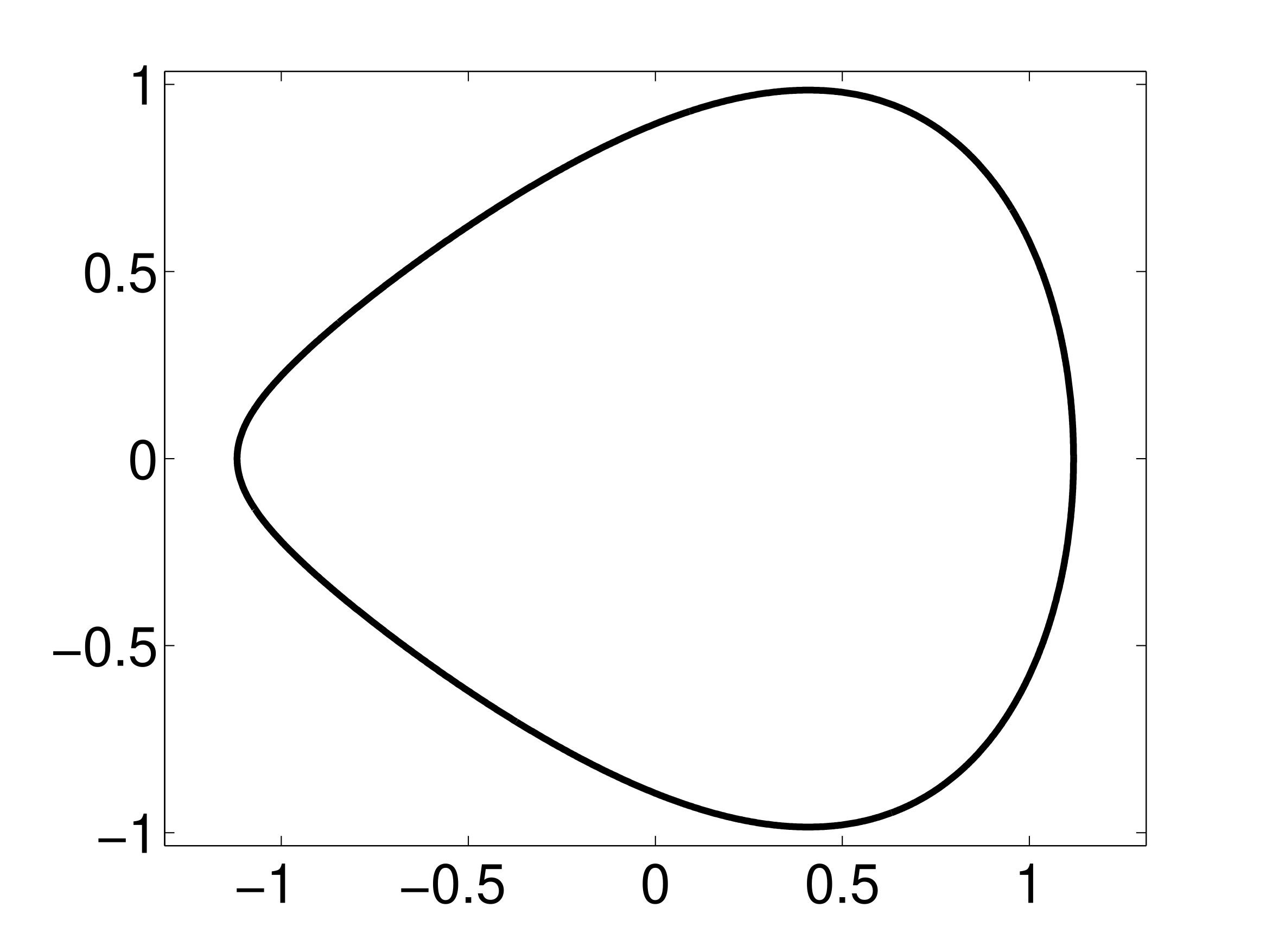}
\includegraphics[width=0.32\textwidth]{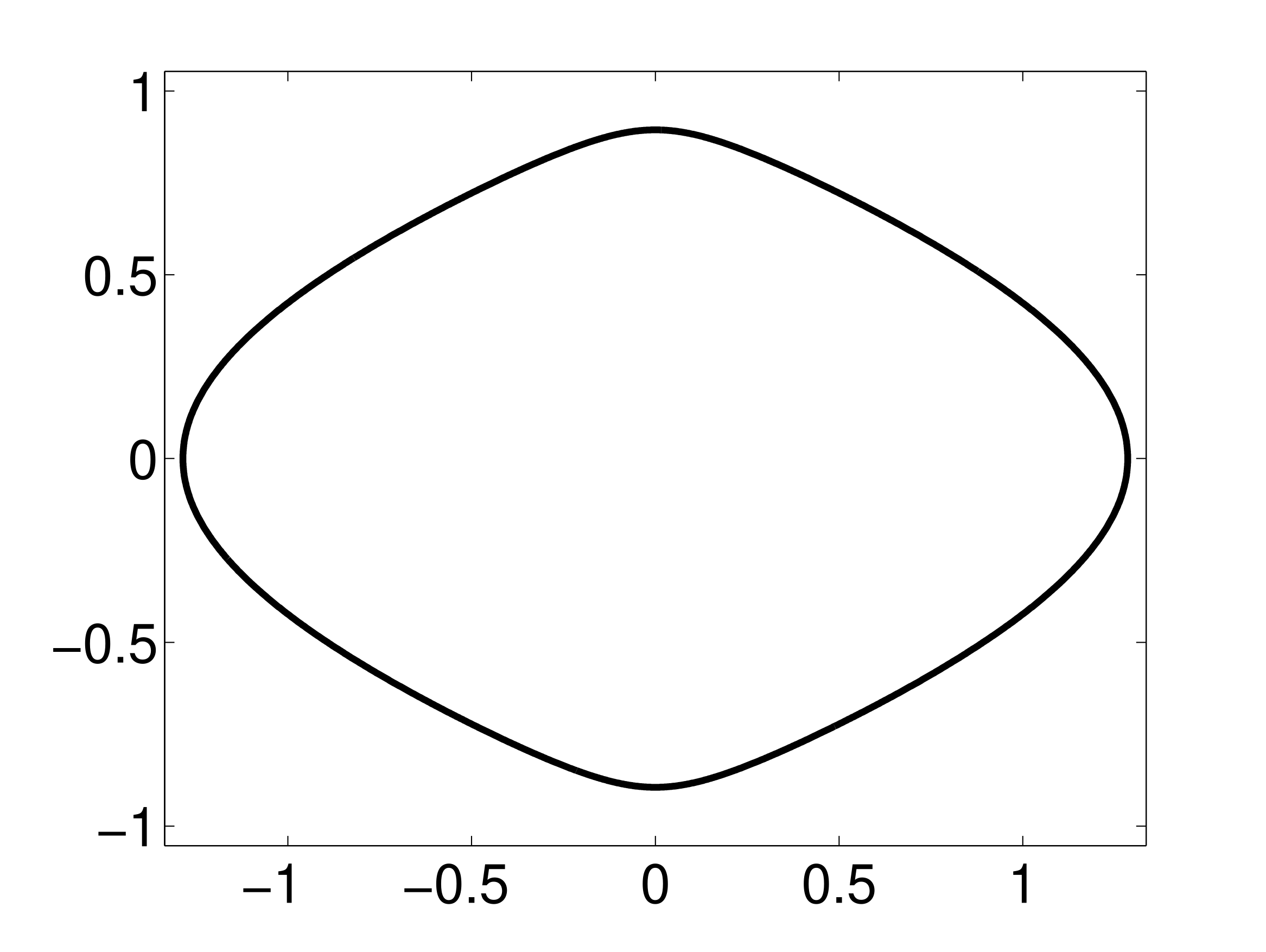}
\includegraphics[width=0.32\textwidth]{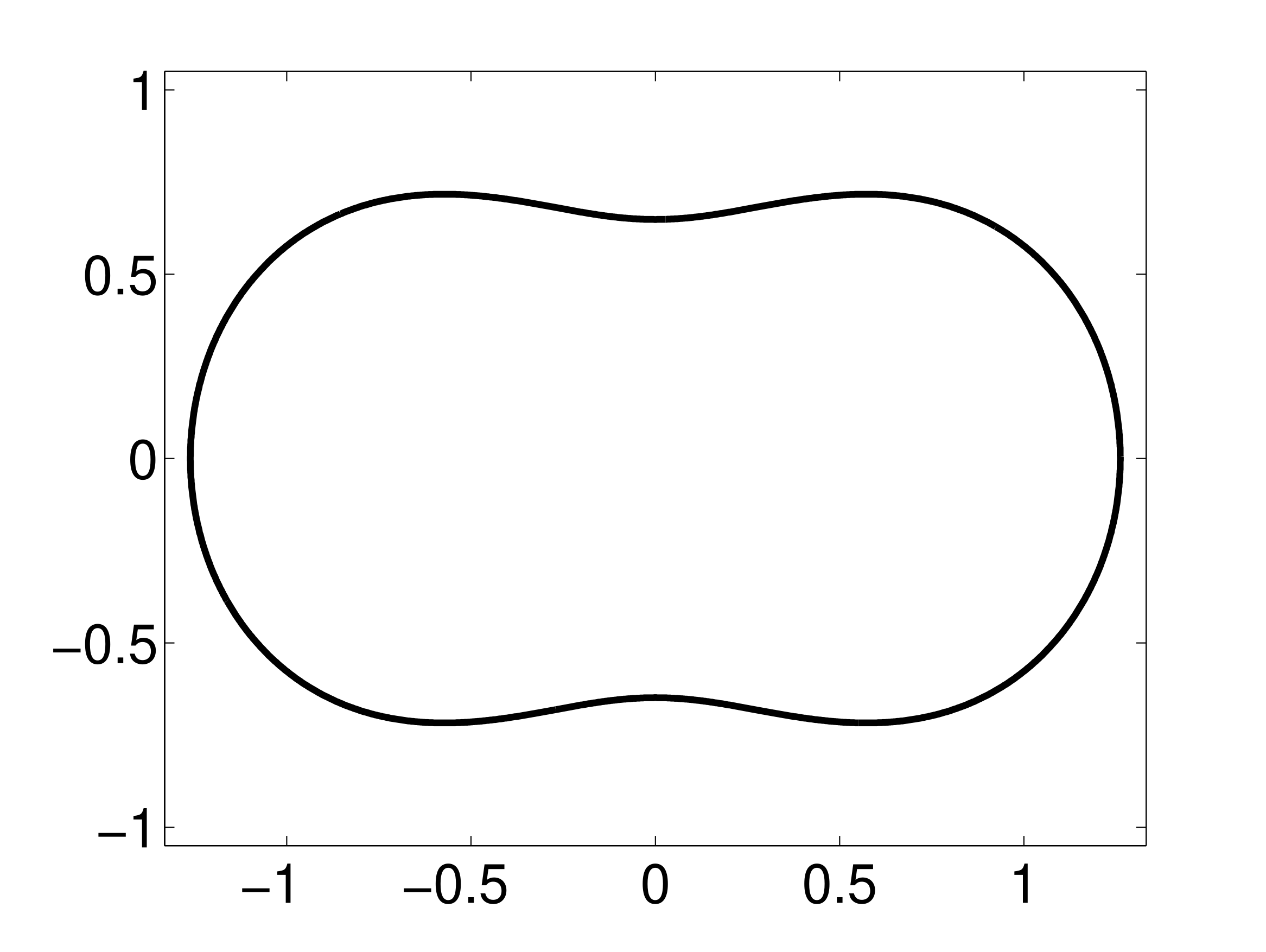}
\caption{Plots of domains $\Omega_1$, $\Omega_2$ and $\Omega_3$.}
\label{fig:figure3}
\end{figure}

\begin{figure}[!ht]
\includegraphics[width=0.9\textwidth]{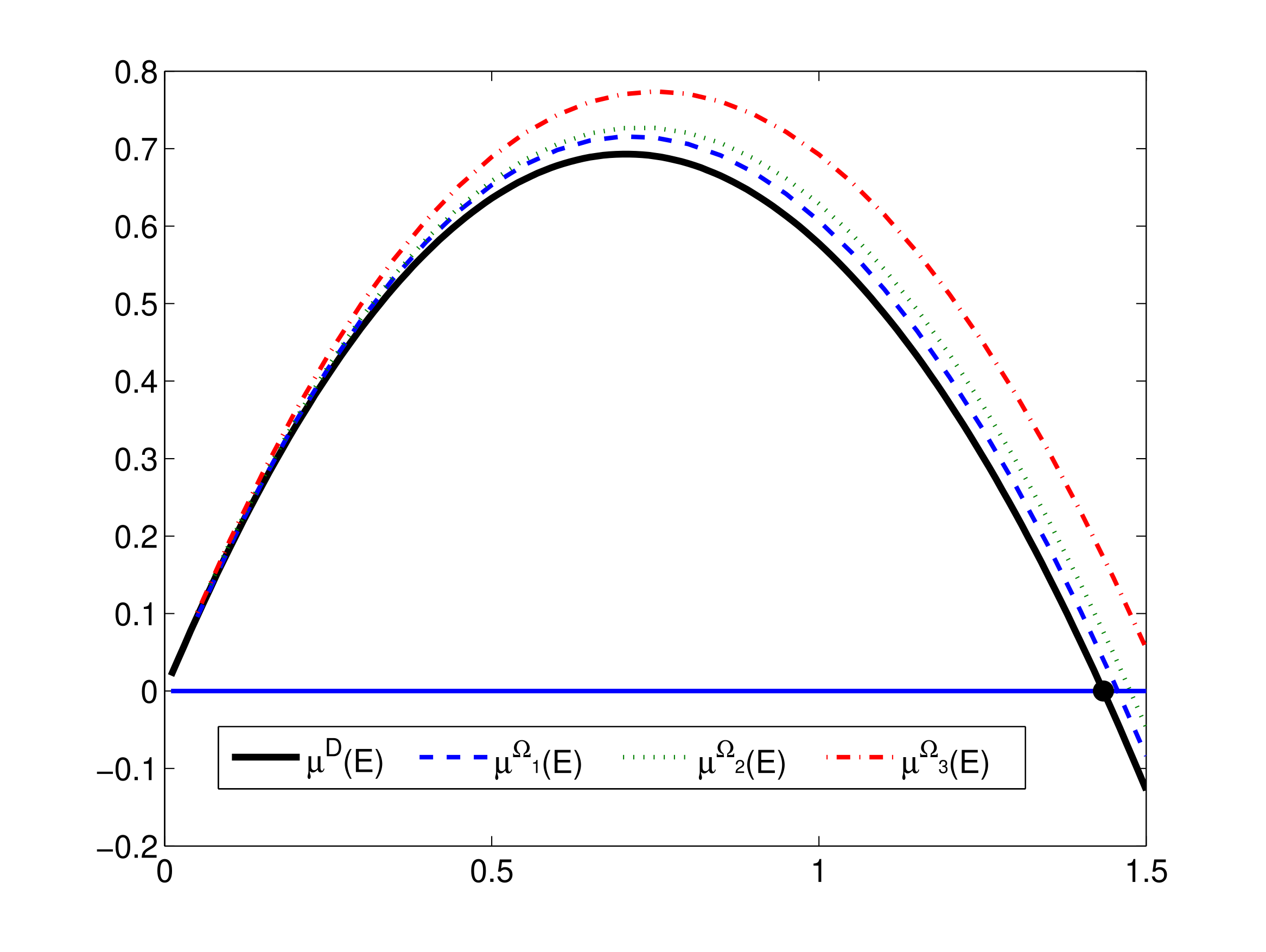}
\caption{Plots of $\mu^{\Omega_i},\ i=1,2,3$, together with the curve $\mu^{\mathbb{D}}$ as a function of the spectral parameter $E>0$.}
\label{fig:figure4}
\end{figure}

Finally, Figure~\ref{fig:figure5} shows the absolute value (left plots) and argument (right plots) of a (normalized) eigenfunction associated to the principal eigenvalue of the domains $\Omega_i,\ i=1,2,3$. Remark that the point of maximal modulus seems to be localized at the incenter of $\Omega_i$ which is in line with our choice of test function in the proof of Theorem \eqref{thm:ineq}. However, there is absolutely no reason for the associated eigenfunction to be real-valued and this has two consequences. First, Theorem \ref{thm:ineq} could be improved if one considers an adequate test function in the domain of the operator and not only in the form domain as we do. Second, Conjecture \eqref{conj:FK} can not be reduced to the Bossel-Daners inequality because, contrary to the Robin eigenvalue problem, there is \textit{a priori} no reason for an eigenfunction to have a non-constant argument as illustrated in Figure~\ref{fig:figure5}.

\begin{figure}[!ht]
\includegraphics[width=0.49\textwidth]{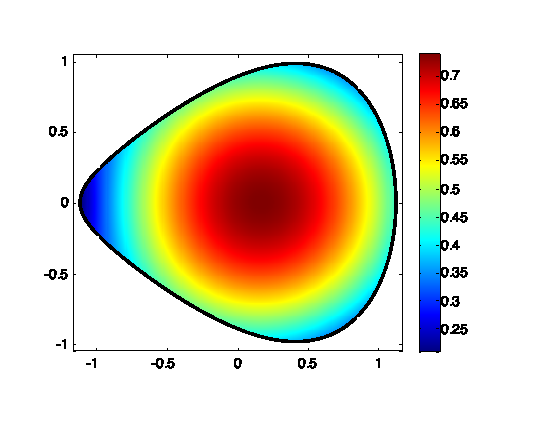}
\includegraphics[width=0.49\textwidth]{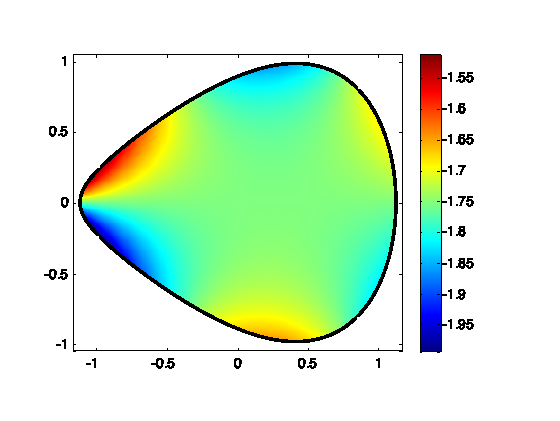}
\includegraphics[width=0.49\textwidth]{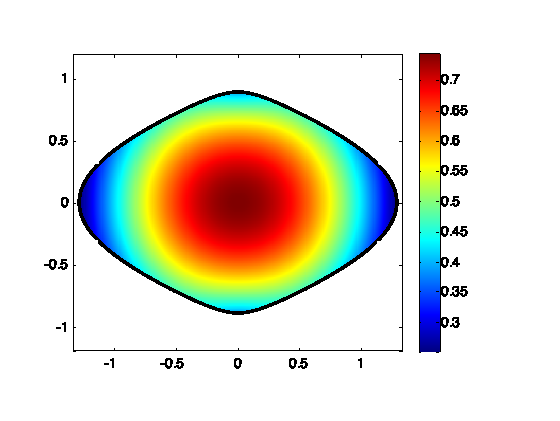}
\includegraphics[width=0.49\textwidth]{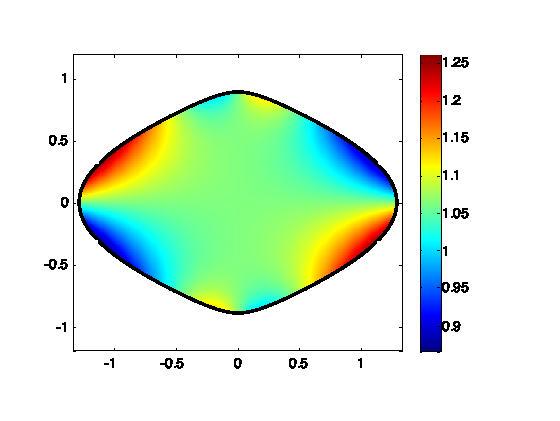}
\includegraphics[width=0.49\textwidth]{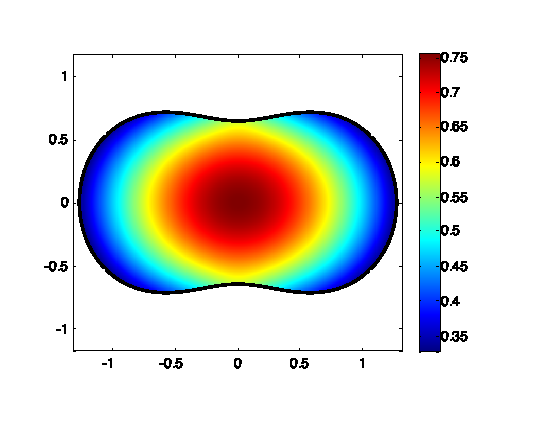}
\includegraphics[width=0.49\textwidth]{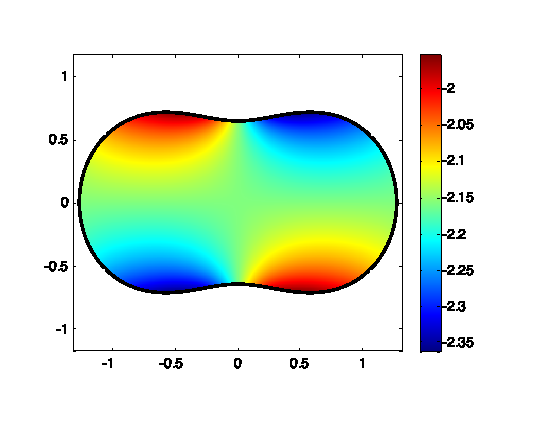}
\caption{Plots of the absolute value (left plots) and argument (right plots) of the eigenfunction associated to the principal eigenvalue of $\Omega_i,\ i=1,2,3$.}
\label{fig:figure5}
\end{figure}
\clearpage
\section*{Acknowledgments}
The work of R.~D.~Benguria has been partially supported by FONDECYT (Chile) project 116-0856.

R.~D.~Benguria, V.~Lotoreichik and T.~Ourmi\`eres-Bonafos are very grateful to the American Institute of Mathematics (AIM) for supporting their participation to the AIM workshop \emph{Shape optimization with surface interactions} in 2019, where this project was initiated.

T.~Ourmi\`eres-Bonafos thanks Nicolas Raymond for pointing out that the projectors introduced in Definition \ref{def:szeproj} are named after the famous mathematician G\'abor Szeg\"o.

\end{document}